\documentclass[reqno]{amsart}
\usepackage{amsmath}
\usepackage{amssymb}
\usepackage{graphicx}
\usepackage{color}

\newtheorem{theorem}{Theorem}[section]
\newtheorem{proposition}[theorem]{Proposition}
\newtheorem{corollary}[theorem]{Corollary}
\newtheorem{lemma}[theorem]{Lemma}
\theoremstyle{definition}
\newtheorem{definition}[theorem]{Definition}
\newtheorem{example}[theorem]{Example}

\newcommand{\oline}[1]{\mathbin{\overline{#1}}}
\newcommand{\uline}[1]{\mathbin{\underline{#1}}}

\newcommand{\overeq}[2]{\overset{ {\rm #1}\ref{#2}}{=}}
\newcommand{\darrow}{\hspace{-0.8mm}\searrow\hspace{-0.8mm}}

\begin{document}

\title{Shadow biquandles and local biquandles}

\author[K.~Oshiro]{Kanako Oshiro}
\address{Department of Information and Communication Sciences, Sophia University, Tokyo 102-8554, Japan}
\email{oshirok@sophia.ac.jp}



\date{\today}

\maketitle

\begin{abstract}
Given a shadow biquandle $(B,X)$ composed of a biquandle $B$ and a strongly connected $B$-set $X$,  we have a local biquandle structure on $X$.  
The (co)homology groups of such shadow biquandles are isomorphic to those of the corresponding local biquandles. 
Moreover, cocycle invariants, of oriented links and oriented surface-links,  using such shadow biquandles coincide with those using the corresponding local biquandles.
These results imply that for some cases, the Niebrzydowski's theory in \cite{Niebrzydowski0, Niebrzydowski1, Niebrzydowski2} for knot-theoretic ternary quasigroups is the same as  shadow biquandle theory.
We also show that some local biquandle $2$- or $3$-cocycles and some $1$- or $2$-cocycles of the Niebrzydowski's (co)homology theory can be induced from Mochizuki's cocycles.

\end{abstract}

\section{Introduction}

In knot theory and related topics, quandles \cite{Joyce82, Matveev82}  are important algebraic systems, each of which equips a binary operation coming from Reidemeister moves of oriented link diagrams with arc labelings.  
Biquandles \cite{FennRourkeSanderson92, KR02} are a generalization of quandles and they are also important algebraic systems, each of which equips two binary operations coming from Reidemeister moves of oriented link diagrams with semi-arc labelings.  As other (or further) important generalizations, shadow quandle theory and shadow biquandle theory were introduced and well-studied, see \cite{CarterKamadaSaito01,  FennRourkeSandersonpreprint, KKKL} for example. Both of them are related to region labelings for oriented link diagrams in addition to arc or semi-arc labelings.  
In each of quandle theory, biquandle theory, shadow quandle theory and shadow biquandle theory, a (co)homology theory and a cocycle invariant of oriented links (or oriented surface-links, oriented virtual-links and so on) using a cocycle of the (co)homology theory are defined and well-studied.

In \cite{Niebrzydowski0, Niebrzydowski1, Niebrzydowski2}, Niebrzydowski studied an algebraic system, called a {\it knot-theoretic ternary quasigroup}, which equips a ternary operation coming from Reidemeister moves of oriented link diagrams with region labelings. He defined a (co)homology theory of the algebraic systems and a cocycle invariant of oriented links and oriented surface-links using a cocycle of his (co)homology theory, see also \cite{ChoiNeedellNelson17, KimNelson, NeedellNelson16}.
Note that the region labelings in this case are not related to arc or semi-arc labelings, while the region labelings in the cases of shadow quandles or shadow biquandles depend on the arc or semi-arc labelings.

In \cite{NOO}, local biquandle theory was introduced. A local biquandle is always given associated with a knot-theoretic (horizontal- or vertical-)ternary-quasigroup. 
Although a local biquandle is not a biquandle, it has a local algebraic structure that is similar to the algebraic structure of biquandles, that is, it has a local algebraic structure related to semi-arc labelings.
It was shown that the Niebrzydowski's (co)homology theory can be interpreted as local biquandle theory. 
On other words, in some sense, the Niebrzydowski's (co)homology theory can be interpreted similarly as biquandle (co)homology theory since local biquandle (co)homology theory is an analogy of biquandle (co)homology theory.
Furthermore, through the interpretation of the (co)homology theories, it was shown that the Niebrzydowski's cocycle invariants and the local biquandle cocycle invariants of oriented links and oriented surface-links are the same. This implies that in some sense, the Niebrzydowski's cocycle invariants can be also interpreted similarly as the biquandle cocycle invariants.

In this paper, we show that 
given a shadow biquandle $(B,X, \uline{*}, \oline{*}, *)$ with a strongly connected $B$-set $X$, we can define a knot-theoretic (horizontal-)ternary-quasigroup $(X, [\,])$  (see Theorem~\ref{thm:1}), and then, we have a local biquandle $(X, \{\star\}, \{\star\})$ associated with $(X, [\,])$.
The (co)homology groups of $(B,X, \uline{*}, \oline{*}, *)$ are isomorphic to those of the corresponding local biquandle $(X, \{\star\}, \{\star\})$ (see Theorem~\ref{thm:2}), and the shadow biquandle cocycle invariant using $(B,X, \uline{*}, \oline{*}, *)$ and a cocycle coincides with the local biquandle cocycle invariant using the corresponding local biquandle $(X, \{\star\}, \{\star\})$ and a corresponding cocycle (see Theorems~\ref{thm:3} and \ref{thm:4}).
Considering the main results in this paper together with the results shown in \cite{NOO}, we can say that for some cases, the Niebrzydowski's theory in \cite{Niebrzydowski0, Niebrzydowski1, Niebrzydowski2}  is the same as  shadow biquandle theory (see Corollaries~\ref{cor:1}, \ref{cor:2} and \ref{cor:3}). 
As a consequence, we show that some local biquandle $2$- or $3$-cocycles and some $1$- or $2$-cocycles of the Niebrzydowski's (co)homology theory can be induced from Mochizuki's cocycles in \cite{Mochizuki} (see Examples~\ref{ex:Mochizuki}).

The paper is organized as follows: 
In Section~2, we review the definitions of links, surface-links, biquandles, horizontal-tribrackets and local biquandles. 
In Section~3, we recall the definitions of shadow biquandle (co)homology groups, shadow biquandle cocycle invariants, local biquandle (co)homology groups and local biquandle cocycle invariants. 
The main results (Theorems~\ref{thm:1}, \ref{thm:2}, \ref{thm:3} and \ref{thm:4}) in this paper are stated and proven in Section~4. 
In Section~5, we give a relationship between the results in this paper and the Niebrzydowski's theory given in \cite{Niebrzydowski0, Niebrzydowski1, Niebrzydowski2}, and we show some local biquandle cocycles and some cocycles of the Niebrzydowski's (co)homology theory that are induced from Mochizuki's cocycles.

\section{Preliminaries}
\subsection{Links, surface-links,  connected diagrams}
A {\it knot} is an oriented $1$-dimensional sphere embedded in $\mathbb R^3$. A {\it link} is a disjoint union of knots.  We note that every knot is a link. 
Two links are said to be {\it equivalent} if they can be deformed into each other through an isotopy of $\mathbb R^3$.
A {\it diagram} of a link is its image by a regular projection, from $\mathbb R^3$ to $\mathbb R^2$, equipped with the height information for each double point.
It is known that two link diagrams represent the same link if and only if they are related by a finite sequence of Reidemeister moves.
A knot diagram is always connected.  
A link diagram with at least two components is said to be {\it connected} if every component intersects another  component.
It is known that between two connected  diagrams $D$ and $D'$ that represent the same link, there exists a finite sequence of connected diagrams and oriented Reidemeister moves that transforms $D$ to $D'$, i.e., there exists 
\[
D=D_0 \overset{R_0} {\longrightarrow} D_1 \overset{R_1} {\longrightarrow}\cdots \overset{R_{i-1}} {\longrightarrow}D_i \overset{R_i} {\longrightarrow} \cdots   \overset{R_{n-1}} {\longrightarrow} D_n =D',
\]  
where for each $i\in \{0,1,\ldots ,n\}$, $R_i$ is an oriented Reidemeister move, and $D_i$ is a connected diagram of a link.
For a diagram $D$, we remove a small neighborhood of each crossing, and then, we call each connected component a {\it semi-arc} of $D$.  
In this paper, for a link diagram $D$, $\mathcal{SA}(D)$ means the set of semi-arcs of $D$ and $\mathcal{R}(D)$ means  the set of connected regions of $\mathbb R^2\setminus D$. 
For a semi-arc $s$ of a link diagram $D$, we assign a normal vector $n_s$ to $s$ to satisfy that 
the pair $(o, n_s)$ of the orientation $o$ of $D$ and $n_s$ coincides with the right-handed orientation of $\mathbb R^2$, and thus, we represent the orientation of $D$.

A {\it surface-knot} is an oriented closed surface locally flatly embedded in $\mathbb R^4$. A {\it surface-link} is a disjoint union of surface-knots.  We note that every surface-knot is a surface-link. 
Two surface-links are said to be {\it equivalent} if they can be deformed into each other through an isotopy of $\mathbb R^4$.
A {\it diagram} of a surface-link is its image by a regular projection, from $\mathbb R^4$ to $\mathbb R^3$, equipped with the height information for each double point curve, where the height information is represented by removing small  neighborhoods of lower double point curves. Then a diagram is composed of  four kinds of local pictures depicted in Figure~\ref{multiplepoint}, and the indicated points are called a {\it regular point}, a {\it double point}, a {\it triple  point} and a {\it branch point}, respectively.  
It is known that two surface-link diagrams represent the same surface-link if and only if they are related by a finite sequence of Roseman moves, see \cite{Roseman} for details.
A surface-knot diagram is always connected.  
A surface-link diagram with at least two components is said to be {\it connected} if every component intersects another  component.
It is known that between two connected  diagrams $D$ and $D'$ that represent the same surface-link, there exists a finite sequence of connected diagrams and oriented Roseman moves that transforms $D$ to $D'$.
For a surface-link diagram $D$, we remove small  neighborhoods of double point curves, and then, we call each connected component a {\it semi-sheet} of $D$.  
In this paper, for a surface-link diagram $D$, $\mathcal{SS}(D)$ means the set of semi-sheets of $D$ and $\mathcal{R}(D)$ means  the set of connected regions of $\mathbb R^3 \setminus D$.
For a semi-sheet $s$ of a surface-link diagram $D$, we assign a normal vector $n_s$ to $s$ to satisfy that 
the triple $(o_1, o_2, n_s)$ of the orientation $(o_1, o_2)$ of $D$ and $n_s$ coincides with the right-handed orientation of $\mathbb R^3$, and thus, we represent the orientation of $D$.      
\begin{figure}[ht]
  \begin{center}
    \includegraphics[clip,width=10cm]{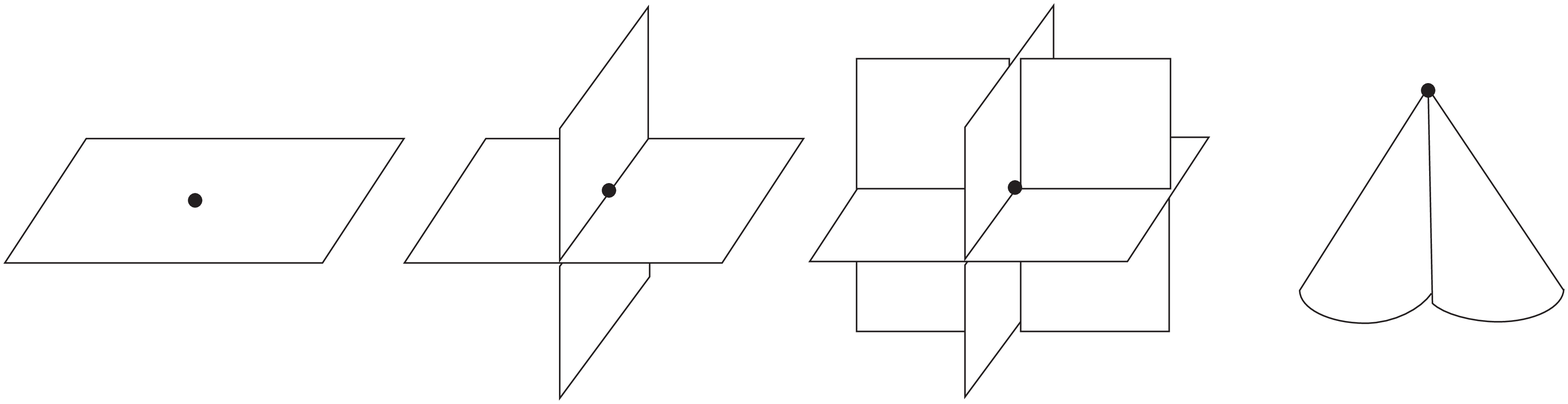}
    \caption{}
    \label{multiplepoint}
  \end{center}
\end{figure}

\subsection{Biquandles, tribrackets, local biquandles}

\begin{definition}\label{def:biquandle}(\cite{FennRourkeSanderson92, KR02})
A \textit{biquandle} is a  set $B$ equipped with binary operations $\uline{*},\oline{*}:B\times B\to B$ satisfying the following axioms.
\begin{itemize}
\item
For any $a\in B$, $a\uline{*}a=a\oline{*}a$.
\item
For any $b\in B$, the map $\uline{*}b:B\to B$ sending $a$ to $a\uline{*}b$ is bijective.
\item[]
For any $b\in B$, the map $\oline{*}b:B\to B$ sending $a$ to $a\oline{*}b$ is bijective.
\item[]
The map $S:B\times B\to B\times B$ defined by $S(a,b)=(b\oline{*}a,a\uline{*}b)$ is bijective.
\item
For any $a,b,c\in B$,
\begin{align*}
&(a\uline{*}b)\uline{*}(c\uline{*}b)=(a\uline{*}c)\uline{*}(b\oline{*}c), \\
&(a\uline{*}b)\oline{*}(c\uline{*}b)=(a\oline{*}c)\uline{*}(b\oline{*}c), \\
&(a\oline{*}b)\oline{*}(c\oline{*}b)=(a\oline{*}c)\oline{*}(b\uline{*}c).
\end{align*}
\end{itemize}
We denote it by $(B, \uline{*}, \oline{*})$ or by $B$ for short unless it causes confusion.
For $a,b\in B$, we denote by $a\uline{*}^{-1} b$ and $a\oline{*}^{-1} b$ the elements $(\uline{*} b)^{-1} (a)$ and $(\oline{*} b)^{-1} (a)$, respectively.
A biquandle $(B, \uline{*}, \oline{*})$ with $a\oline{*} b = a ~(\forall a,b \in B)$ is called a {\it quandle} \cite{Joyce82, Matveev82}, which is also denoted by $(B, \uline{*})$.
\end{definition}

\begin{definition}\label{def:B-set}
Let $(B, \uline{*}, \oline{*})$ be a biquandle. 

\begin{itemize}
\item[(1)] A {\it $B$-set} is a  set $X$ equipped with a map $*: X\times B \to X$ satisfying the following axiom:

\begin{itemize}
\item[$\bullet$]
For any $a \in B$, $(*a): X \to X$ sending  $x$  to $x*a$ is bijective.
\item[$\bullet$]
For any $a, b\in B$ and  $x\in X$, $(x*a)*(b\oline{*}a)=(x*b)*(a \uline{*}b)$.
\end{itemize}

\noindent We denote it by $(X,*)$  or by $X$ for short unless it causes confusion.
We denote by $x *^{-1} a$ the element $(*a)^{-1} (x)$ for $a\in B$ and $x\in X$.

\item[(2)]  A $B$-set $(X,*)$ is {\it strongly connected} if  it satisfies the following axiom:

\begin{itemize}
\item[$\bullet$] For any $x\in X$, the map $x*: B \to  X$ sending $a$ to $x*a$ is bijective.
\end{itemize}

\noindent We denote  by $x \darrow y$ the element $(x*)^{-1}(y)$ for $x,y\in B$.

\end{itemize}
\end{definition}

\begin{definition}
A {\it shadow biquandle} is a pair of a biquandle $(B, \uline{*}, \oline{*})$ and a $B$-set $(X,*)$. We denote it by $(B,X,\uline{*}, \oline{*}, *)$ or by $(B,X)$ for short unless it causes confusion.
In particular, when $(B, \uline{*}, \oline{*})$ is a quandle, the shadow biquandle $(B,X,\uline{*}, \oline{*}, *)$ is also called a {\it shadow quandle} and denoted  by $(B, X, \uline{*}, *)$. 
\end{definition}
%
%

\begin{lemma}\label{lem:0}
Let $(B, X, \uline{*}, \oline{*}, *)$ be a shadow biquandle.
For any $x\in X$ and $a,b\in B$, we have
\begin{itemize}
\item[(1)] $\big(x*^{-1}  (a\uline{*} b)\big) *^{-1} b = \big(x*^{-1} (b\oline{*} a)\big) *^{-1} a$,
\item[(2)] $(x*^{-1}a) *  (b\oline{*}^{-1} a)  = (x*b)*^{-1} \big(a \uline{*} (b\oline{*}^{-1} a)\big)$,
\item[(3)] $(x*^{-1}b) *  (a \uline{*}^{-1} b)  = (x*a)*^{-1} \big(b \oline{*} (a\uline{*}^{-1} b)\big)$.
\end{itemize}
\end{lemma}
\begin{proof}
We leave the proof of this lemme to the reader, refer also to Figure~\ref{lem1}.
\begin{figure}[h]
  \begin{center}
    \includegraphics[clip,width=10cm]{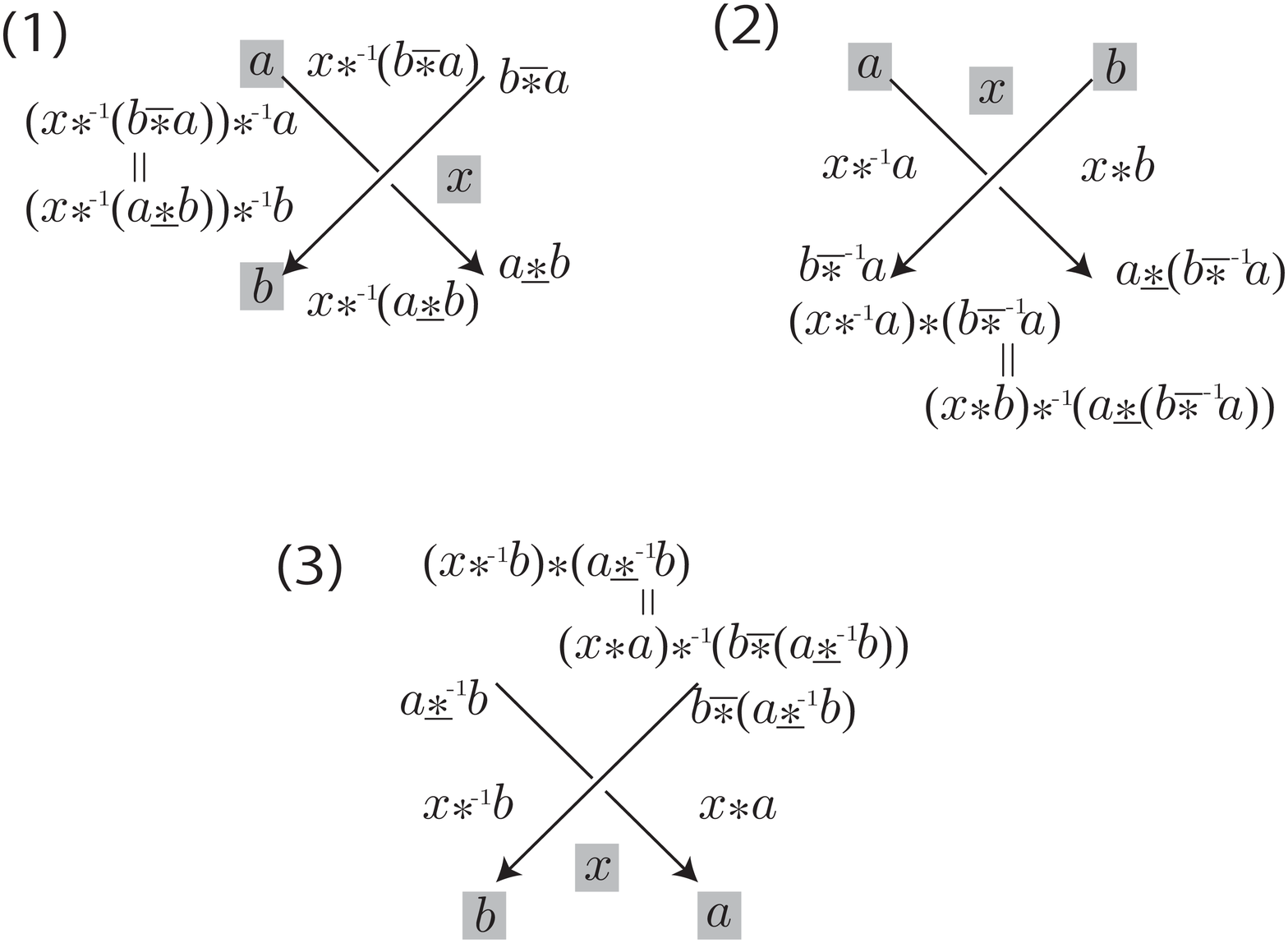}
    \caption{}
    \label{lem1}
  \end{center}
\end{figure}
\end{proof}
\begin{lemma}\label{lem:1}
Let $(B, X, \uline{*}, \oline{*}, *)$ be a shadow biquandle such that $X$ is strongly connected.
For $x,y\in X$ and $a,b \in B$, we have 
\begin{itemize}
\item[(1)]  $x*(x \darrow  y)=y$, 
\item[(2)]  $y*^{-1}(x \darrow  y)=x$, 
\item[(3)]  $(y*^{-1}a) \darrow  y=a$, 
\item[(4)]  $x \darrow  (x*a)=a$.
\end{itemize}
\end{lemma}
\begin{proof}
This lemma follows from the above definitions, refer also to Figure~\ref{lem2}.
\begin{figure}[h]
  \begin{center}
    \includegraphics[clip,width=8cm]{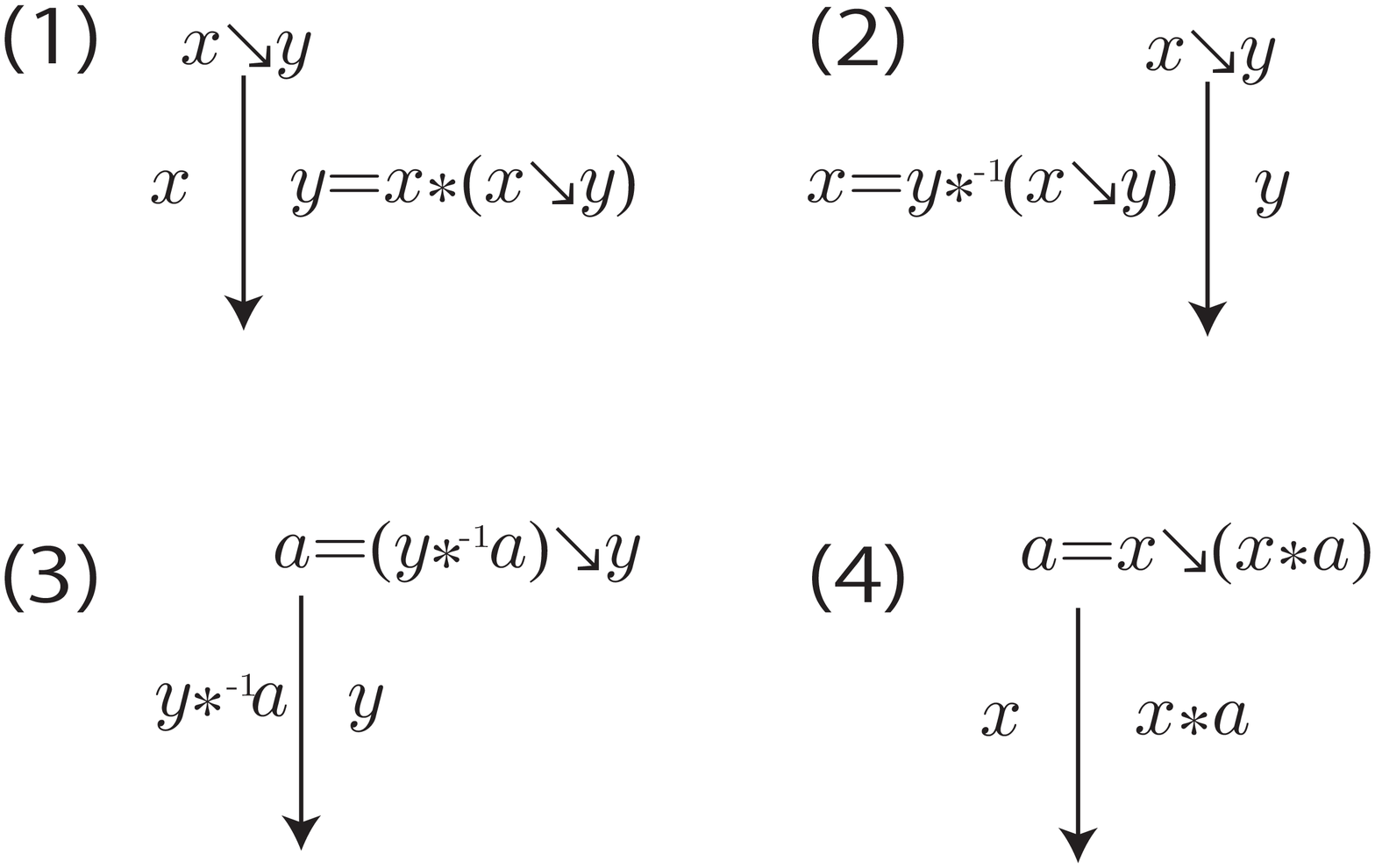}
    \caption{}
    \label{lem2}
  \end{center}
\end{figure}
\end{proof}

\begin{lemma}\label{lem:latin}
Let $(Q,\uline{*})$ be a latin quandle, that is, it satisfies that
\begin{itemize}
\item for any $a\in Q$, $a\uline{*}: Q \to Q$ is bijective.
\end{itemize}
Let $X=Q$ and let $*:  X\times Q \to Q$ be defined by  $*= \uline{*}$.
Then $(X, *)$ is a strongly connected $Q$-set.
\end{lemma}
\begin{proof} This can be easily shown  by direct observation. We leave the proof of this lemma to the reader.
\end{proof}

In this paper, for a positive integer $n$, $\mathbb Z_n$ means the quotient ring  $\mathbb Z/ n\mathbb Z$, and $\mathbb Z_n [t^{\pm 1}]$ means the Laurent polynomial ring with coefficients in $\mathbb Z_n$.
\begin{example}\label{ex:1}
For a positive integer $n$, let $R_n=\mathbb Z_n$. 
We define $\uline{*}: R_n^2 \to R_n$ by 
$a \uline{*} b = 2b-a$, 
and then, $(R_n , \uline{*})$ is a quandle called the {\it dihedral quandle} of order $n$. 
Let $X=\mathbb Z_n$. Then $X$ is an $R_n$-set with $*: X \times R_n \to X$ defined by $*=\uline{*}$.

In particular when $n$ is an odd number other than $1$, since $R_n$ is latin, $X$ is strongly connected by Lemma~\ref{lem:latin}. We then have $x \darrow y = 2^{-1}(x+y)$ for $x,y \in X$.
\end{example}
\noindent The next example is a generalization of Example~\ref{ex:1}.
\begin{example}\label{ex:2}
For a positive integer $n$ and an ideal $J$ of $\mathbb Z_n [t^{\pm 1}]$, let $Q=\mathbb Z_n [t^{\pm 1}]/J$ be the quotient ring.
We define $\uline{*}: Q^2 \to Q$ by  $a \uline{*} b = ta+(1-t)b$, 
and then, $(Q, \uline{*})$ is a quandle called an {\it Alexander quandle}.
Let $X=Q$. Then $X$ is a $Q$-set with $*: X \times Q \to X$ defined by  $*=\uline{*}$.

In particular when $1-t$ is a unit in $\mathbb Z_n [t^{\pm 1}]/J$, since $Q$ is latin, $X$ is strongly connected by Lemma~\ref{lem:latin}. We then have $x \darrow y = (1-t)^{-1}(-tx+y)$ for $x,y \in X$.
\end{example}

\begin{definition}\label{def:horizontal} (cf. \cite{Niebrzydowski1, NOO})
 A {\it knot-theoretic horizontal-ternary-quasigroup}  is a pair of  a  set $X$  and a ternary operation $[\, ]: X^3 \to X; (x,y,z) \mapsto [x,y,z]$ satisfying the following property:
\begin{enumerate}
\item[] \hspace{-0.5cm}($\mathcal{H}$1)  
\begin{itemize}
\item[(i)] For any $x,y,w\in X$, there exists a unique $z\in X$ such that $[x,y,z]=w$, 
\item[(ii)] For any $x,z,w\in X$, there exists a unique $y\in X$ such that $[x,y,z]=w$, 
\item[(iii)] For any $y,z,w\in X$, there exists a unique $x\in X$ such that $[x,y,z]=w$.
\end{itemize}
\item[] \hspace{-0.5cm}($\mathcal{H}$2) For any $x,y,z,w \in X$, it holds that 
\[
\begin{array}{l}
[y,[x,y,z],[x,y,w]] = [z,[x,y,z],[x,z,w]]=[w,[x,y,w],[x,z,w]].
\end{array} 
\]
\end{enumerate}
We call the operation $[\,]$  a {\it horizontal-tribracket}.
\end{definition}


\begin{definition}(\cite{NOO})
Let $(X,[\,])$ be a knot-theoretic horizontal-ternary-quasigroup.
For each $x \in X$, we define two operations $\uline{\star}_x, \oline{\star}_x: (\{x\} \times X)^2 \to X^2$ by 
\[
\begin{array}{l}
(x,y) \uline{\star}_x (x,z) = (z, [x,y,z]), \mbox{ and }\\[5pt]
(x,y) \oline{\star}_x (x,z) = (z, [x,z,y]). 
\end{array}
\]
We call $(X, \{\uline{\star}_x\}_{x\in X}, \{\oline{\star}_x\}_{x\in X})$ the {\it local biquandle associated with $(X,[\, ])$}.
In this paper, for simplicity, we often omit the subscript by $x$ as $\uline{\star}=\uline{\star}_x$, $\oline{\star}=\oline{\star}_x$,  
$\{\uline{\star}\}=\{\uline{\star}_x\}_{x\in X} $,  and 
$\{\oline{\star}\}=\{\oline{\star}_x\}_{x\in X} $ unless it causes confusion. 
\end{definition}

The next two examples are related to Examples~\ref{ex:1} and \ref{ex:2}, respectively, which will be shown in Subsection~\ref{subsec:Mochizuki}.
\begin{example}\label{ex:3}
For a positive integer $n$, let $X=\mathbb Z_n$. We define a map $[\,]: X^3 \to X$ by 
\[
[x,y,z] = x-y+z,
\]  
and then, $[\,]$ is a horizontal-tribracket. We call it the {\it dihedral horizontal-tribracket} of order $n$.

The local biquandle $(X, \{\uline{\star}\}, \{\oline{\star}\})$ associated with this $(X, [\,])$
has the operations $\uline{\star}, \oline{\star}: X^2 \to X$ defined by 
\[
\begin{array}{l}
(x,y) \uline{\star} (x,z) = (y, x-y+z), \mbox{ and }\\[5pt]
(x,y) \oline{\star} (x,z) = (z,x+y-z). 
\end{array}
\]
\end{example}

\begin{example}\label{ex:4}
For a positive integer $n$ and an ideal $J$ of $\mathbb Z_n [t^{\pm 1}]$, let $X=\mathbb Z_n [t^{\pm 1}]/J$ be the quotient ring.
We define a map $[\,]: X^3 \to X$ by 
\[
[x,y,z] = -tx + ty + z,
\]  
and then, $[\,]$ is a horizontal-tribracket. 
We call it an {\it Alexander horizontal-tribracket}.

The local biquandle $(X, \{\uline{\star}\}, \{\oline{\star}\})$ associated with this $(X, [\,])$
has the operations $\uline{\star}, \oline{\star}: X^2 \to X$ defined by 
\[
\begin{array}{l}
(x,y) \uline{\star} (x,z) = (z, -tx +ty + z), \mbox{ and }\\[5pt]
(x,y) \oline{\star} (x,z) = (z, -tx  + y +tz). 
\end{array}
\]

\end{example}

\section{Local biquandle homology groups/cocycle invariants and shadow biquandle homology groups/cocycle invariants}

\subsection{Shadow biquandle homology groups}
 Let $(B, X, \uline{*}, \oline{*}, *)$ be a shadow biquandle.

Let $C_n^{\rm sb}(B, X)$ be the free $\mathbb Z$-module generated by the elements of $X\times B^{n}$
if $n\geq 1$, and $C_n^{\rm sb}(B, X)=0$ otherwise.
We define a homomorphism $\partial_n^{\rm sb} : C_n^{\rm sb}(B, X) \to C_{n-1}^{\rm sb}(B, X)$ by 
\begin{align}
&\partial_n^{\rm sb} \big( (x, a_1, \ldots , a_n) \big)  \notag \\
&= \sum_{i=1}^{n} (-1)^i \Big\{ \big( x, a_1, \ldots , a_{i-1}, a_{i+1}, \ldots , a_n \big)\notag  \\
&\hspace{2cm}- \big( x * a_i, a_1 \uline{*} a_i, \ldots , a_{i-1} \uline{*} a_i , a_{i+1}\oline{*} a_i, \ldots , a_n \oline{*} a_i \big) \Big\} \notag 
\end{align}
if $n> 0$, and $\partial_n^{\rm sb}=0$ otherwise.
 Then 
$C_*^{\rm sb}(B, X)=\{C_n^{\rm sb}(B, X), \partial_n^{\rm sb}\}_{n\in \mathbb Z}$ is a chain complex.
Let $D_n^{\rm sb}(B, X)$ be a submodule of $C_n^{\rm sb}(B, X)$ that is generated by the elements of 
\[
\Big\{\big( x,a_1, \ldots , a_n \big) \in X\times B^n  ~\Big|~ \mbox{ $a_i =a_{i+1}$ for some $i\in \{1, \ldots , n-1 \} $  }  \Big\}.
\]
Then
$D_*^{\rm sb}(B, X)=\{D_n^{\rm sb}(B, X), \partial_n^{\rm sb}\}_{n\in \mathbb Z}$ is a subchain complex of $C_*^{\rm sb}(B, X)$. 
Therefore the chain complex $$C_*^{\rm SB} (B, X)=\{C_n^{\rm SB}(B, X):=C_n^{\rm sb}(B, X)/D_n^{\rm sb}(B, X), \partial_n^{\rm SB}:=\partial_n^{\rm sb}\}_{n\in \mathbb Z}$$ is induced.
We call the homology group $H_n^{\rm SB} (B, X)$ of $C_*^{\rm SB} (B, X)$ the \textit{$n$th shadow biquandle homology group} of  $(B, X)$.

For an abelian group $A$, we define the chain and cochain complexes by 
\[
\begin{array}{l}
C_n^{\rm SB}(B, X; A)=C_n^{\rm SB}(B, X) \otimes A, \quad \partial_n^{\rm SB} \otimes {\rm id}  \mbox{ and }\\[5pt]
C_{\rm SB}^n(B, X; A) ={\rm Hom}(C_n^{\rm SB}(B, X); A), \quad \delta^n_{\rm SB} \mbox{ s.t. }\delta^n_{\rm SB}(f)=f \circ \partial_{n+1}^{\rm SB}.
\end{array}
\]
Let $C_\ast^{\rm SB}(B, X; A)=\{C_n^{\rm SB}(B, X; A), \partial_n^{\rm SB}\otimes {\rm id}\}_{n\in \mathbb Z}$ and $C_{\rm SB}^\ast(B, X; A)=\{C_{\rm SB}^n(B, X; A), \delta^n_{\rm SB}\}_{n\in \mathbb Z}$. The \textit{nth homology group} $H_n^{\rm SB}(B, X; A)$ \textit{and nth cohomology group} $H^n_{\rm SB}(B, X; A)$ of $(B, X ) $  with coefficient group $A$ are defined by
\[
H_n^{\rm SB}(B, X; A)=H_n(C_\ast^{\rm SB}(B, X; A)) \qquad {\rm and} \qquad H_{\rm SB}^n(B, X; A)=H^n(C^\ast_{\rm SB}(B, X; A)).
\]
The {\it $n$th cocycle group} with coefficient group $A$  is denoted by $Z^n_{\rm SB}(B, X; A)$. 
Note that we omit the coefficient group  $A$ if $A=\mathbb Z$ as usual.

\subsection{Shadow biquandle colorings of link diagrams and cocycle invariants}\label{subsection:linkinvariant}
 Let $(B, X, \uline{*}, \oline{*}, *)$ be a shadow biquandle.
Let $D$ be a  diagram of a link $L$.
\begin{definition}
A \textit{ (biquandle) $B$-coloring} of $D$ is a map $C: \mathcal{SA}(D)  \to  B$ satisfying the following condition:
\begin{itemize}
\item For a crossing composed of under-semi-arcs $u_1, u_2$ and over-semi-arcs $o_1, o_2$ as depicted in Figure~\ref{coloring6}, 
\begin{itemize}
\item $C(u_2) = C(u_1) \uline{*} C(o_1)$, and 
\item $C(o_2) = C(o_1) \oline{*} C(u_1)$
\end{itemize}
hold, see also Figure~\ref{coloring9}.
\end{itemize}
\end{definition}

\begin{definition}\label{def:coloringshadow}
A \textit{ (shadow biquandle) $(B,X)$-coloring} of $D$ is a map $C: \mathcal{SA}(D) \cup \mathcal{R}(D)  \to  B \cup X$ satisfying the following condition:
\begin{itemize}
\item The restriction $C|_{\mathcal{SA}(D)}$ is a $B$-coloring of $D$.
\item $C(\mathcal{R}(D))\subset X$.
\item For a semi-arc $s$ whose normal vector points from a region $r_1$ to a region $r_2$  as depicted in Figure~\ref{coloring6},  $C(r_1) *
C(s) = C(r_2)$ holds, see also Figure~\ref{coloring9}.
\end{itemize}
We denote by ${\rm Col}_{(B,X)}^{\rm SB} (D) $ the set of $(B,X)$-colorings of $D$. 
\begin{figure}[h]
  \begin{center}
    \includegraphics[clip,width=10.0cm]{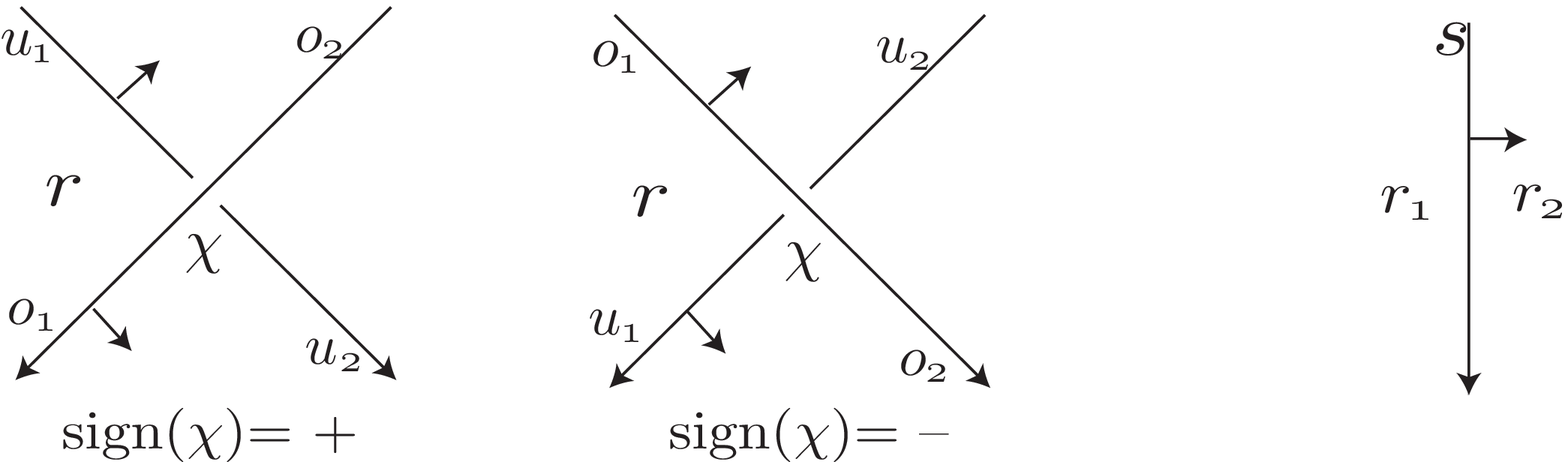}
    \caption{}
    \label{coloring6}
  \end{center}
\end{figure}
\begin{figure}[h]
  \begin{center}
    \includegraphics[clip,width=10.0cm]{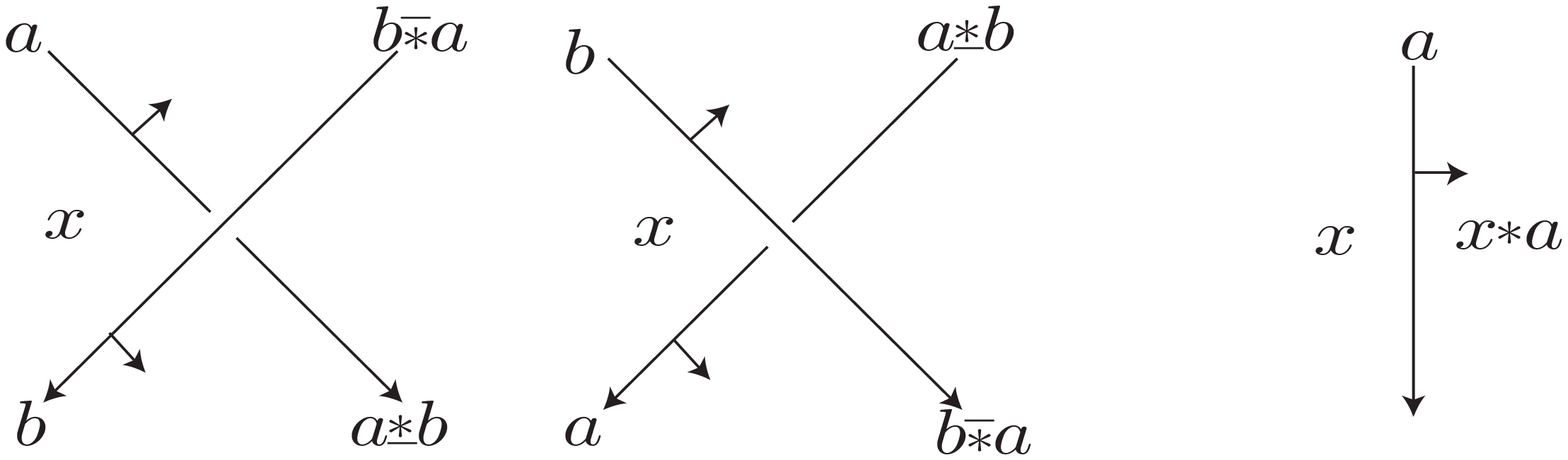}
    \caption{}
    \label{coloring9}
  \end{center}
\end{figure}
\end{definition}

\begin{proposition}\label{prop:coloring1}{\rm (cf. \cite{KKKL})}
Let $D$ and $D'$ be connected diagrams of links. 
If $D$ and $D'$ represent the same link, then there exists a bijection between  
${\rm Col}_{(B,X)}^{\rm SB} (D) $ and ${\rm Col}_{(B,X)}^{\rm SB} (D')  $.
\end{proposition}

Next, we show how to obtain a cocycle invariant by using the  $(B,X)$-colorings of a  diagram.

Let $C$ be a  $(B,X)$-coloring of $D$.
We define the local chain 
$w^{\rm SB}(D, C; \chi) \in C^{\rm SB}_2 (B, X)$ at each crossing $\chi$ by 
$$w^{\rm SB}(D, C; \chi) ={\rm sign}(\chi) \big(x,a,b \big)$$ when  $C(r)=x$, $C(u_1)=a$ and $C(o_1)=b$, where $r$, $u_1$ and $o_1$ are the region, under-semi-arc and over-semi-arc of $\chi$  as depicted in Figure~\ref{coloring6}, see also Figure~\ref{coloring9}..
We define a chain by 
$$\displaystyle W^{\rm SB}(D, C)=\sum_{\chi \in \{\mbox{\small crossings of $D$}\}} w^{\rm SB}(D, C; \chi) \in C^{\rm SB}_2 (B, X).$$

Let $A$ be an abelian group. For a $2$-cocycle $\theta \in C^2_{\rm SB}(B, X; A)$, we define
\[
\begin{array}{l}
\mathcal{H}^{\rm SB}(D)=\{[W^{\rm SB}(D, C)] \in H^{\rm SB}_2(B, X) \ | \ C \in {\rm Col}_{(B,X)}^{\rm SB} (D) \}, and \\[5pt]
\Phi_{\theta}^{\rm SB}(D)=\{\theta(W^{\rm SB}(D, C)) \in A \ | \ C \in {\rm Col}_{(B,X)}^{\rm SB} (D) \}
\end{array}
\]
as multisets. Then we have the following theorem:
\begin{theorem}{\rm (cf. \cite{KKKL})}
$\mathcal{H}^{\rm SB}(D)$ and $\Phi_{\theta}^{\rm SB}(D)$ are invariants of $L$.
\end{theorem}

\subsection{Shadow biquandle colorings of surface-link diagrams, cocycle invariants}

 Let $(B, X, \uline{*}, \oline{*}, *)$ be a shadow biquandle.
Let $D$ be a  diagram of a surface-link $F$.
\begin{definition}
A \textit{ (biquandle) $B$-coloring} of $D$ is a map $C: \mathcal{SS}(D)  \to  B$ satisfying the following condition:
\begin{itemize}
\item For a double point curve composed of under-semi-sheets $u_1, u_2$ and over-semi-sheets $o_1, o_2$ as depicted in Figure~\ref{doublepointshadow}, 
\begin{itemize}
\item $C(u_2) = C(u_1) \uline{*} C(o_1)$, and 
\item $C(o_2) = C(o_1) \oline{*} C(u_1)$
\end{itemize}
hold, see also Figure~\ref{doublepointshadow2}.
\end{itemize}
\end{definition}

\begin{definition}\label{def:coloringshadow}
A \textit{ (shadow biquandle) $(B,X)$-coloring} of $D$ is a map $C: \mathcal{SS}(D) \cup \mathcal{R}(D)  \to  B \cup X$ satisfying the following condition:
\begin{itemize}
\item The restriction $C|_{\mathcal{SS}(D)}$ is a $B$-coloring of $D$.
\item $C(\mathcal{R}(D))\subset X$.
\item For a semi-sheet $s$ whose normal vector points from a region $r_1$ to a region $r_2$  as depicted in Figure~\ref{doublepointshadow},  $C(r_1) *
C(s) = C(r_2)$ holds, see also Figure~\ref{doublepointshadow2}.
\end{itemize}
We denote by ${\rm Col}_{(B,X)}^{\rm SB} (D) $ the set of $(B,X)$-colorings of $D$. 
\begin{figure}[ht]
  \begin{center}
    \includegraphics[clip,width=8.0cm]{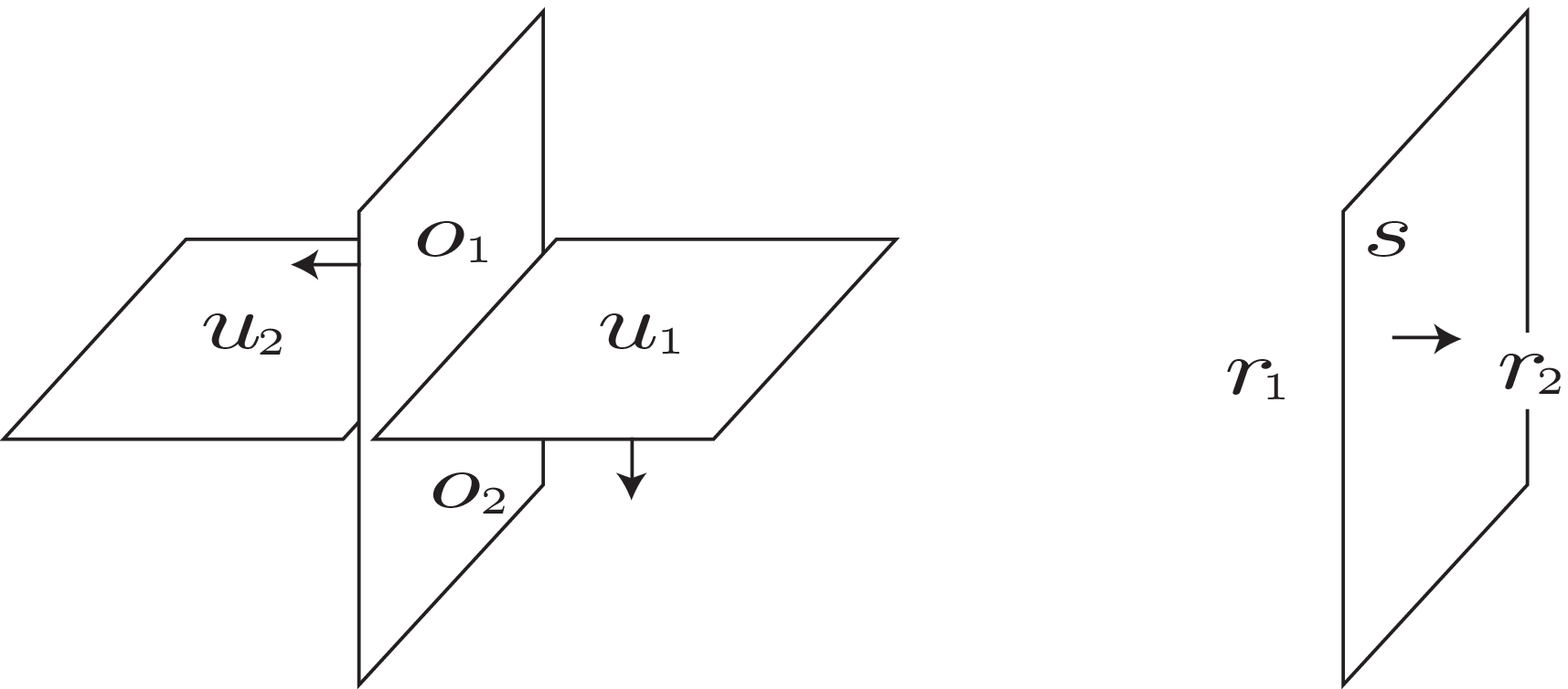}
    \caption{}
    \label{doublepointshadow}
  \end{center}
\end{figure}
\begin{figure}[ht]
  \begin{center}
    \includegraphics[clip,width=8.0cm]{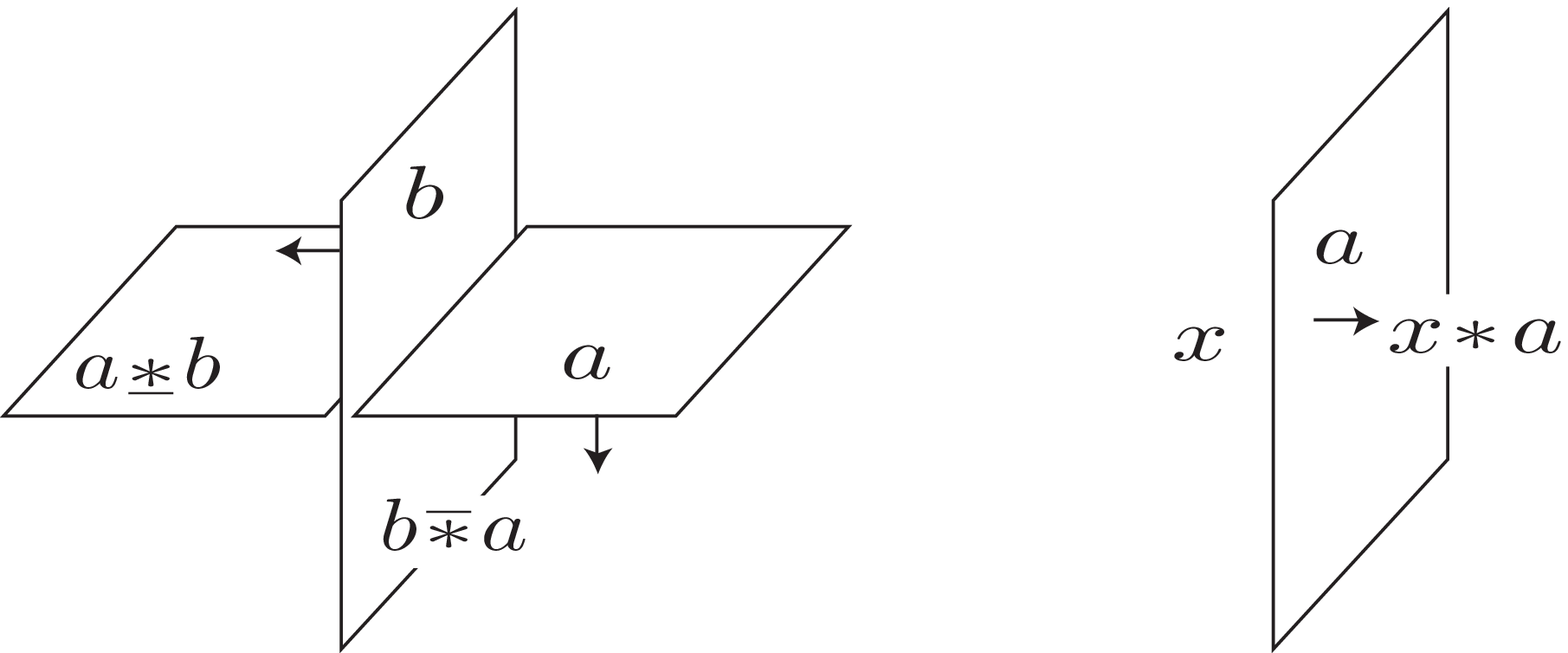}
    \caption{}
    \label{doublepointshadow2}
  \end{center}
\end{figure}
\end{definition}

\begin{proposition}\label{prop:coloring1}{\rm (cf. \cite{KKKL})}
Let $D$ and $D'$ be connected diagrams of surface-links. 
If $D$ and $D'$ represent the same surface-link, then there exists a bijection between  
${\rm Col}_{(B,X)}^{\rm SB} (D) $ and ${\rm Col}_{(B,X)}^{\rm SB} (D')  $.
\end{proposition}

Next, we show how to obtain a cocycle invariant by using the  $(B,X)$-colorings of a  diagram.

Let $C$ be a  $(B,X)$-coloring of $D$.
We define the local chain 
$w^{\rm SB}(D, C; \tau) \in C^{\rm SB}_3 (B, X)$ at each triple point $\tau$ by 
$$w^{\rm SB}(D, C; \tau) ={\rm sign}(\tau) \big(x,a,b, c \big)$$ when  $C(r_1)=x$, $C(b_1)=a$, $C(m_1)=b$ and $C(t_1)=c$, where $r_1$, $b_1$, $m_1$ and $t_1$ are the region, bottom-semi-sheet, middle-semi-sheet and top-semi-sheet of $\tau$   as depicted in Figure~\ref{triplepoint4}, see also Figure~\ref{triplepoint5}.
We define a chain by 
$$\displaystyle W^{\rm SB}(D, C)=\sum_{\tau \in \{\mbox{\small triple points of $D$}\}} w^{\rm SB}(D, C; \tau) \in C^{\rm SB}_3 (B, X).$$ 
\begin{figure}[ht]
  \begin{center}
    \includegraphics[clip,width=8.0cm]{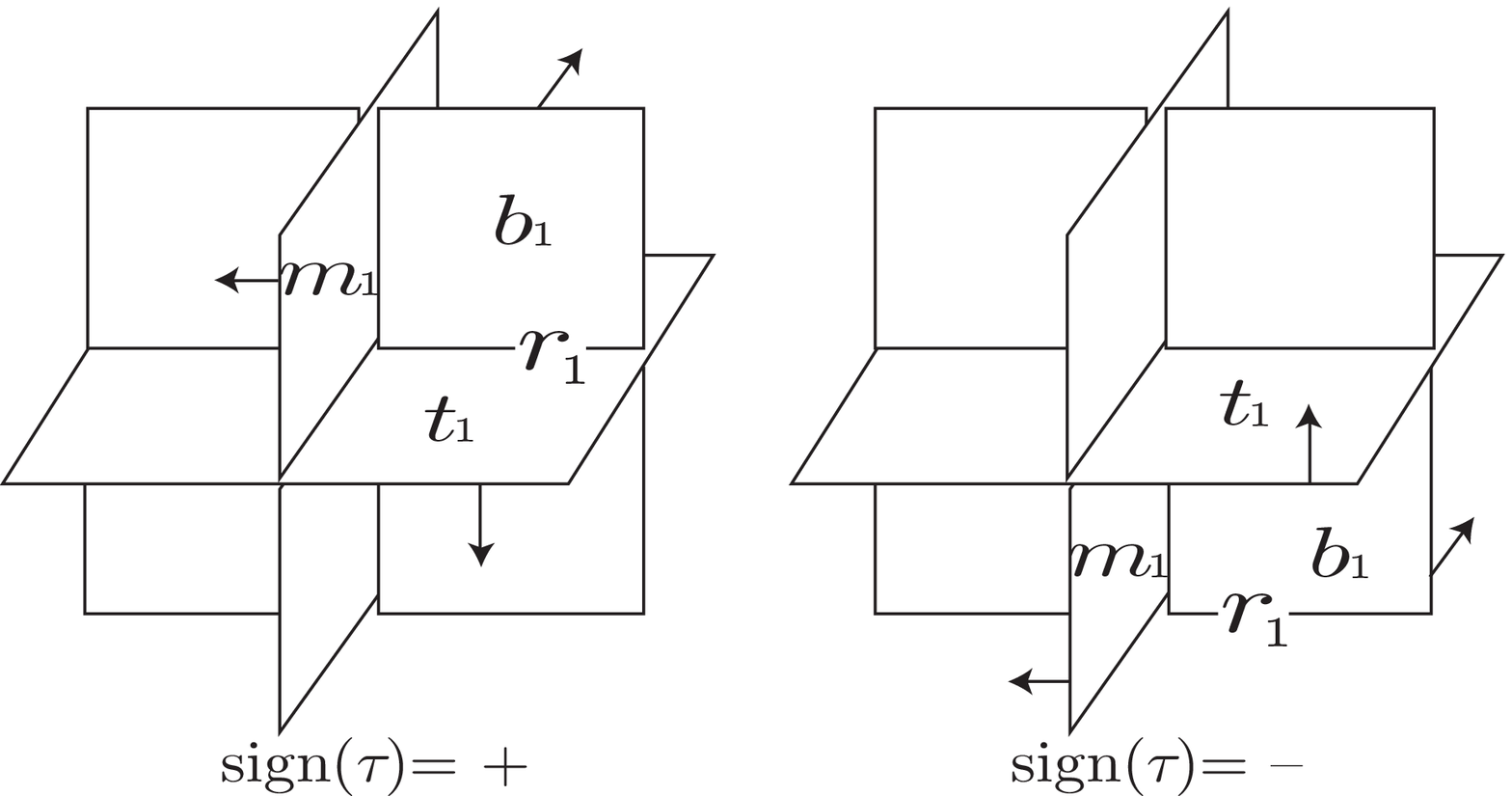}
    \caption{}
    \label{triplepoint4}
  \end{center}
\end{figure}
\begin{figure}[ht]
  \begin{center}
    \includegraphics[clip,width=8.0cm]{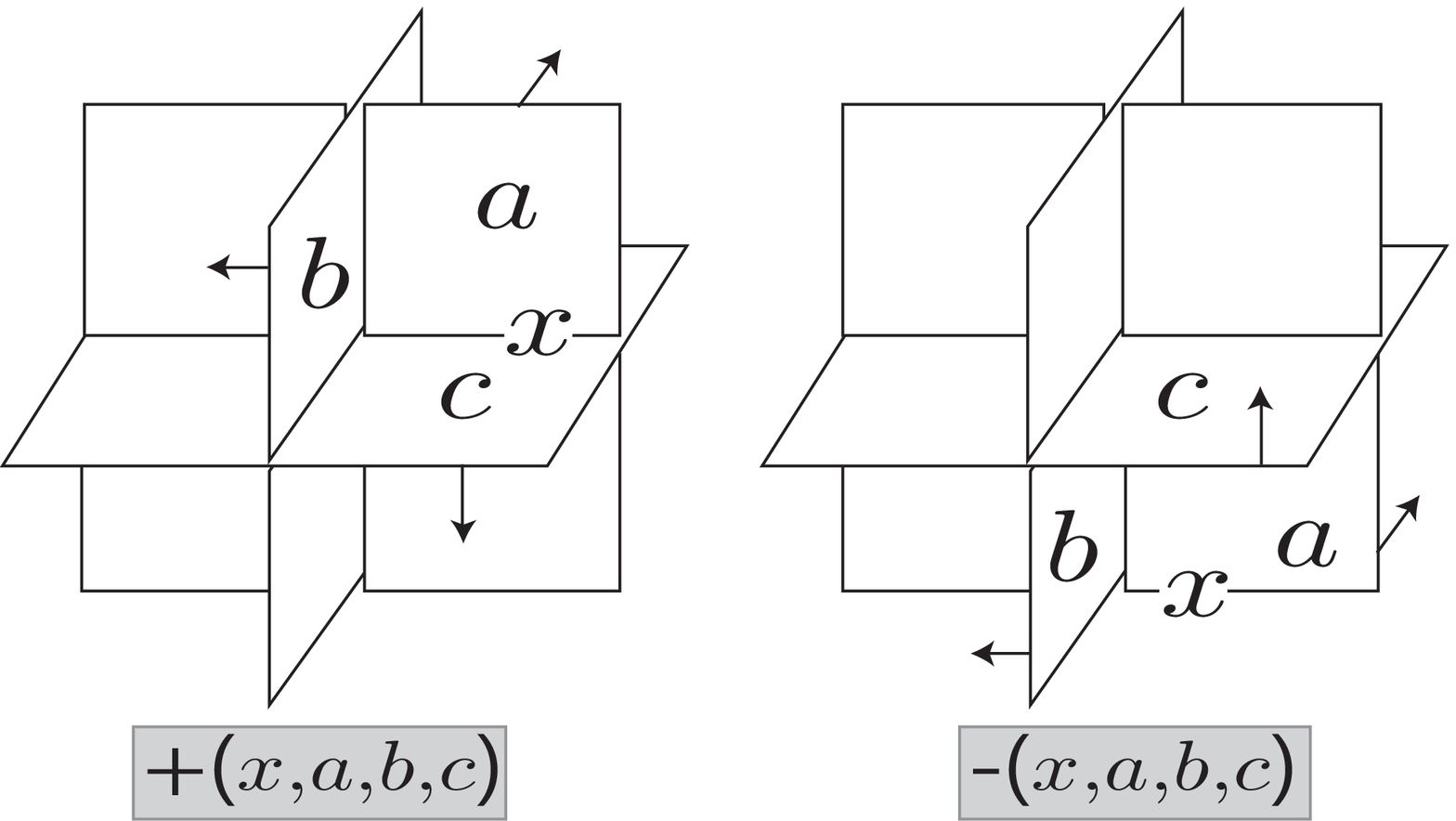}
    \caption{}
    \label{triplepoint5}
  \end{center}
\end{figure}

Let $A$ be an abelian group. For a $3$-cocycle $\theta \in C^3_{\rm SB}(B, X; A)$, we define
\[
\begin{array}{l}
\mathcal{H}^{\rm SB}(D)=\{[W^{\rm SB}(D, C)] \in H^{\rm SB}_3(B, X) \ | \ C \in {\rm Col}_{(B,X)}^{\rm SB} (D) \}, and \\[5pt]
\Phi_{\theta}^{\rm SB}(D)=\{\theta(W^{\rm SB}(D, C)) \in A \ | \ C \in {\rm Col}_{(B,X)}^{\rm SB} (D) \}
\end{array}
\]
as multisets. Then we have the following theorem:
\begin{theorem}{\rm (cf. \cite{KKKL})}
$\mathcal{H}^{\rm SB}(D)$ and $\Phi_{\theta}^{\rm SB}(D)$ are invariants of $F$.
\end{theorem}


\subsection{Local biquandle homology groups}

Let $(X,[\,])$ be a knot-theoretic horizontal-ternary-quasigroup and $(X, \{\uline{\star}\}, \{ \oline{\star}\}) $ the local biquandle associated with $(X,[\,])$.
Let $n \in \mathbb Z$.
Let $C^{\rm lb}_n(X)$ be the free $\mathbb Z$-module generated by the elements of 
\[
\bigcup_{x\in X} (\{x\} \times X)^n =\big\{\big( (x,y_1), (x,y_2), \ldots , (x,y_n) \big) ~|~ x, y_1, \ldots , y_n \in X \big\}
\]
if $n\geq 1$, and $C^{\rm lb}_n(X)=0$ otherwise.
We define a homomorphism $\partial_n^{\rm lb} : C_n^{\rm lb} (X) \to C_{n-1}^{\rm lb} (X)$ by 
\begin{align}
&\partial_n^{\rm lb} \Big( \big( (x,y_1), \ldots , (x,y_n) \big) \Big) = \sum_{i=1}^{n} (-1)^i \big\{ \big( (x,y_1), \ldots, (x,y_{i-1}), (x,y_{i+1}), \ldots  , (x,y_n) \big) \notag \\
&- \big( (x,y_1)\uline{\star} (x,y_i), \ldots, (x,y_{i-1})\uline{\star} (x,y_i),    (x,y_{i+1})\oline{\star} (x,y_i) ,\ldots  , (x,y_n) \oline{\star} (x,y_i)\big) \big\} \notag \\
 &= \sum_{i=1}^{n} (-1)^i \big\{ \big( (x,y_1), \ldots,  (x,y_{i-1}), (x,y_{i+1}),  \ldots  , (x,y_n) \big) \notag \\
&- \big( (y_i, [x,y_1, y_i] ), \ldots, (y_i, [x,y_{i-1}, y_i]),    (y_i, [x,y_i, y_{i+1}]) ,\ldots  , (y_i, [x,y_i, y_{n}])\big) \big\} \notag 
\end{align}
if $n\geq 2$, and $\partial_n^{\rm lb}=0$ otherwise. Then $C_*^{\rm lb}(X)=\{C_n^{\rm lb}(X), \partial_n^{\rm lb}\}_{n\in \mathbb Z}$ is a chain complex.
Let $D_n^{\rm lb}(X)$ be a submodule of $C_n^{\rm lb}(X)$ that is generated by the elements of 
\[
\Big\{\big( (x,y_1),  \ldots , (x,y_n) \big) \in \bigcup_{x\in X
}  (\{x\} \times X)^n ~\Big|~ \mbox{ $y_i =y_{i+1}$ for some $i\in \{1, \ldots , n-1 \} $  }  \Big\}.
\]
Then 
$D_*^{\rm lb}(X)=\{D_n^{\rm lb}(X), \partial_n^{\rm lb}\}_{n\in \mathbb Z}$ is a subchain complex of $C_*^{\rm lb}(X)$.
Therefore the chain complex $$C_*^{\rm LB} (X)=\{C_n^{\rm LB}(X):=C_n^{\rm lb}(X)/D_n^{\rm lb}(X), \partial_n^{\rm LB}:=\partial_n^{\rm lb}\}_{n\in \mathbb Z}$$ is induced.
We call the homology group $H_n^{\rm LB} (X)$ of $C_*^{\rm LB} (X)$ the \textit{$n$th local biquandle homology group} of $(X, \{\uline{\star}\} , \{ \oline{\star}\})$.

For an abelian group $A$, we define the chain and cochain complexes by 
\[
\begin{array}{l}
C_n^{\rm LB}(X; A)=C_n^{\rm LB}(X) \otimes A, \quad \partial_n^{\rm LB} \otimes {\rm id}  \mbox{ and }\\[5pt]
C_{\rm LB}^n(X; A) ={\rm Hom}(C_n^{\rm LB}(X); A), \quad \delta^n_{\rm LB} \mbox{ s.t. }\delta^n_{\rm LB}(f)=f \circ \partial_{n+1}^{\rm LB}.
\end{array}
\]
Let $C_\ast^{\rm LB}(X; A)=\{C_n^{\rm LB}(X; A), \partial_n^{\rm LB}\otimes {\rm id}\}_{n\in \mathbb Z}$ and $C_{\rm LB}^\ast(X; A)=\{C_{\rm LB}^n(X; A), \delta^n_{\rm LB}\}_{n\in \mathbb Z}$. The \textit{nth homology group} $H_n^{\rm LB}(X; A)$ \textit{and nth cohomology group} $H^n_{\rm LB}(X; A)$ of $(X, \{\uline{\star}\}, \{\oline{\star}\})$ with coefficient group $A$ are defined by
\[
H_n^{\rm LB}(X; A)=H_n(C_\ast^{\rm LB}(X; A)) \qquad {\rm and} \qquad H_{\rm LB}^n(X; A)=H^n(C^\ast_{\rm LB}(X; A)).
\]
The {\it $n$th cocycle group} with coefficient group $A$  is denoted by $Z^n_{\rm LB}(X; A)$. 
Note that we omit the coefficient group  $A$ if $A=\mathbb Z$ as usual.

\subsection{Local biquandle colorings of link diagrams, cocycle invariants}\label{subsection:linkinvariant}
Let $(X,[\,])$ be a knot-theoretic horizontal-ternary-quasigroup and $(X, \{\uline{\star}\}, \{ \oline{\star}\}) $ the local biquandle associated with $(X,[\,])$.
Let $D$ be a connected diagram of a link $L$.
\begin{definition}\label{def:localbiquandlecolor}
A \textit{(local biquandle) $X^2$-coloring} of $D$ is a map $C: \mathcal{SA}(D) \to X^2$ satisfying the following condition:
\begin{itemize}
\item For a crossing composed of under-semi-arcs $u_1, u_2$ and over-semi-arcs $o_1, o_2$ as depicted in Figure~\ref{coloring6}, let $C(u_1)=(x_1,y), C(o_1)=(x_2, z)$. Then 
\begin{itemize}
\item $x_1=x_2$,  
\item $C(u_2) = C(u_1) \uline{\star} C(o_1)= (x,y) \uline{\star} (x,z) = (z, [x,y,z])$, and 
\item $C(o_2) = C(o_1) \oline{\star} C(u_1)=(x,z) \oline{\star} (x,y) = (y, [x,y,z])$
\end{itemize}
hold,  where $x= x_1 =x_2$, see also Figure~\ref{coloring1}.
\end{itemize}
We denote by ${\rm Col}_{X^2}^{\rm LB} (D) $ the set of  $X^2$-colorings of $D$.
\end{definition}

\begin{figure}[ht]
  \begin{center}
    \includegraphics[clip,width=10.0cm]{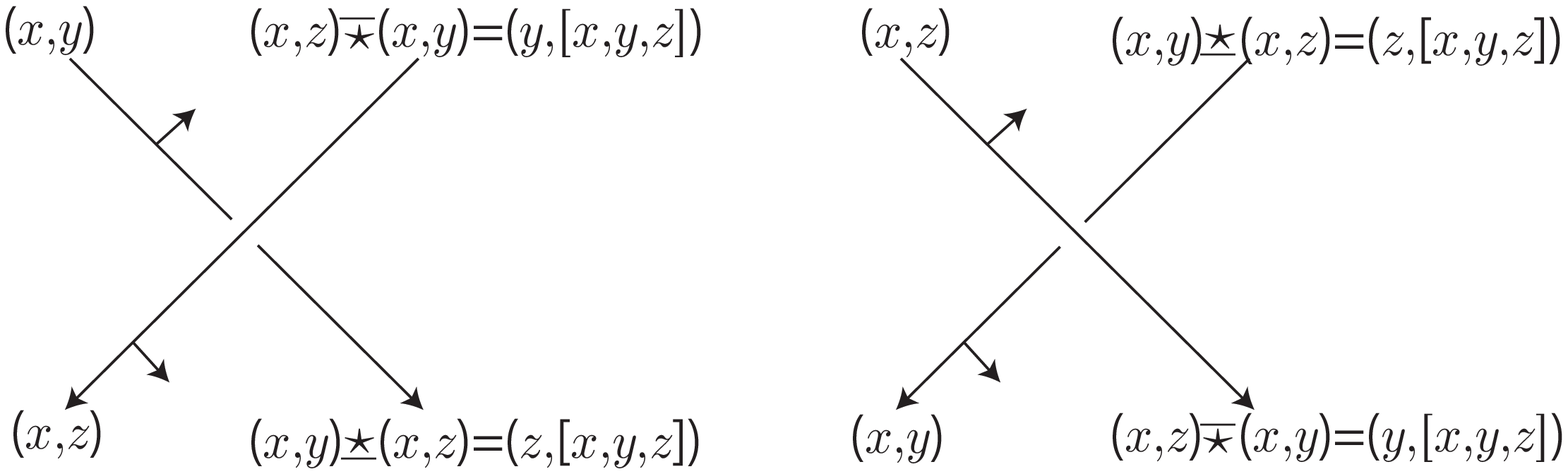}
    \caption{}
    \label{coloring1}
  \end{center}
\end{figure}

\begin{proposition}\label{prop:coloring1}{\rm (\cite{NOO})}
Let $D$ and $D'$ be connected diagrams of links. 
If $D$ and $D'$ represent the same link, then there exists a bijection between  
${\rm Col}_{X^2}^{\rm LB} (D) $ and ${\rm Col}_{X^2}^{\rm LB} (D')  $.
\end{proposition}

Next, we show how to obtain a cocycle invariant by using the $X^2$-colorings of a connected diagram.

Let $C$ be an  $X^2$-coloring of $D$.
We define the local chain 
$w^{\rm LB}(D, C; \chi) \in C^{\rm LB}_2 (X)$ at each crossing $\chi$ by 
$$w^{\rm LB}(D, C; \chi) ={\rm sign}(\chi) \big((x,y), (x,z)\big)$$ when  $C(u_1)=(x,y)$ and $C(o_1)=(x,z)$, where $u_1$ and $o_1$ are the under-semi-arc and over-semi-arc of $\chi$  as depicted in Figure~\ref{coloring6}, see also Figure~\ref{coloring1}.
We define a chain by 
$$\displaystyle W^{\rm LB}(D, C)=\sum_{\chi \in \{\mbox{\small crossings of $D$}\}} w^{\rm LB}(D, C; \chi) \in C^{\rm LB}_2 (X).$$

Let $A$ be an abelian group. For a $2$-cocycle $\theta \in C^2_{\rm LB}(X; A)$, we define
\begin{align*}
&\mathcal{H}^{\rm LB}(D)=\big\{[W^{\rm LB}(D, C)] \in H^{\rm LB}_2(X) \ \big| \ C \in {\rm Col}_{X^2}^{\rm LB} (D) \big\},\mbox{ and }\\ 
&\Phi_{\theta}^{\rm LB}(D)=\big\{\theta\big(W^{\rm LB}(D, C)\big) \in A \ \big| \ C \in {\rm Col}_{X^2}^{\rm LB} (D) \big\}
\end{align*}
as multisets. Then we have the following theorem:
\begin{theorem}{\rm (\cite{NOO})}
$\mathcal{H}^{\rm LB}(D)$ and $\Phi_{\theta}^{\rm LB}(D)$ are invariants of $L$.
\end{theorem}


\subsection{Local biquandle colorings of surface-link diagrams, cocycle invariants}
Let $(X,[\,])$ be a knot-theoretic horizontal-ternary-quasigroup and $(X, \{\uline{\star}\}, \{ \oline{\star}\}) $ the local biquandle associated with $(X,[\,])$.
Let $D$ be a connected diagram of a surface-link $F$.
\begin{definition}\label{def:localbiquandlecolor}
A \textit{(local biquandle) $X^2$-coloring} of $D$ is a map $C: \mathcal{SS}(D) \to X^2$ satisfying the following condition:
\begin{itemize}
\item For a double point curve composed of under-semi-sheets $u_1, u_2$ and over-semi-sheets $o_1, o_2$ as depicted in Figure~\ref{doublepointshadow}, let $C(u_1)=(x_1,y), C(o_1)=(x_2, z)$. Then 
\begin{itemize}
\item $x_1=x_2$,  
\item $C(u_2) = C(u_1) \uline{\star} C(o_1)= (x,y) \uline{\star} (x,z) = (z, [x,y,z])$, and 
\item $C(o_2) = C(o_1) \oline{\star} C(u_1)=(x,z) \oline{\star} (x,y) = (y, [x,y,z])$
\end{itemize}
hold,  where $x= x_1 =x_2$, see also Figure~\ref{doublepointlocal2}.
\end{itemize}
We denote by ${\rm Col}_{X^2}^{\rm LB} (D) $ the set of  $X^2$-colorings of $D$. 
\end{definition}

\begin{figure}[ht]
  \begin{center}
    \includegraphics[clip,width=6.0cm]{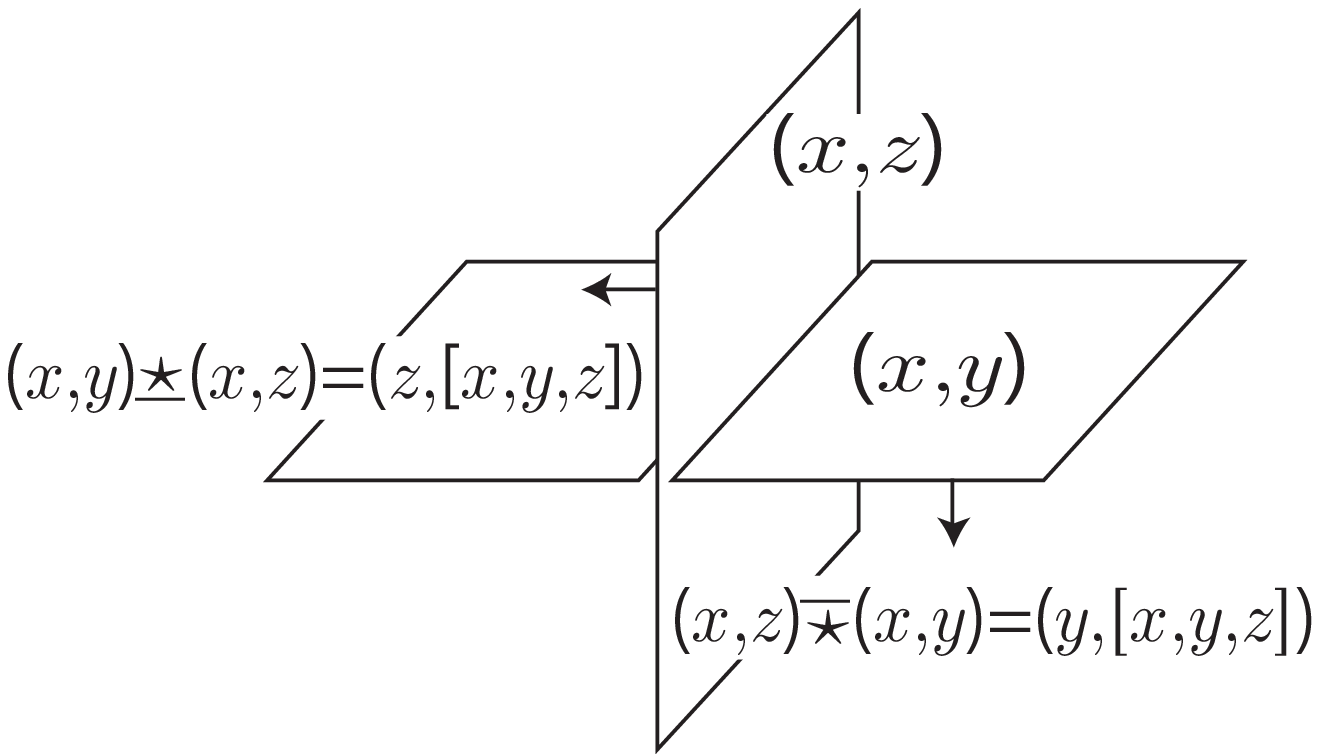}
    \caption{}
    \label{doublepointlocal2}
  \end{center}
\end{figure}

\begin{proposition}\label{prop:coloring1}{\rm (\cite{NOO})}
Let $D$ and $D'$ be connected diagrams of surface-links. 
If $D$ and $D'$ represent the same surface-link, then there exists a bijection between  
${\rm Col}_{X^2}^{\rm LB} (D) $ and ${\rm Col}_{X^2}^{\rm LB} (D')  $.
\end{proposition}

Next, we show how to obtain a cocycle invariant by using the $X^2$-colorings of a connected diagram.

Let $C$ be an  $X^2$-coloring of $D$.
We define the local chain 
$w^{\rm LB}(D, C; \tau) \in C^{\rm LB}_3 (X)$ at each triple point $\tau$ by 
$$w^{\rm LB}(D, C; \tau) ={\rm sign}(\tau) \big((x,y), (x,z), (x,w)\big)$$ when  $C(b_1)=(x,y)$, $C(m_1)=(x,z)$ and $C(t_1)=(x,w)$, where $b_1$, $m_1$ and $t_1$ are the bottom-semi-sheet, middle-semi-sheet and top-semi-sheet of $\tau$ as depicted in Figure~\ref{triplepoint4}, see also Figure~\ref{triplepoint}.
We define a chain by 
$$\displaystyle W^{\rm LB}(D, C)=\sum_{\tau \in \{\mbox{\small triple points of $D$}\}} w^{\rm LB}(D, C; \tau) \in C^{\rm LB}_3 (X).$$ 

\begin{figure}[ht]
  \begin{center}
    \includegraphics[clip,width=8.0cm]{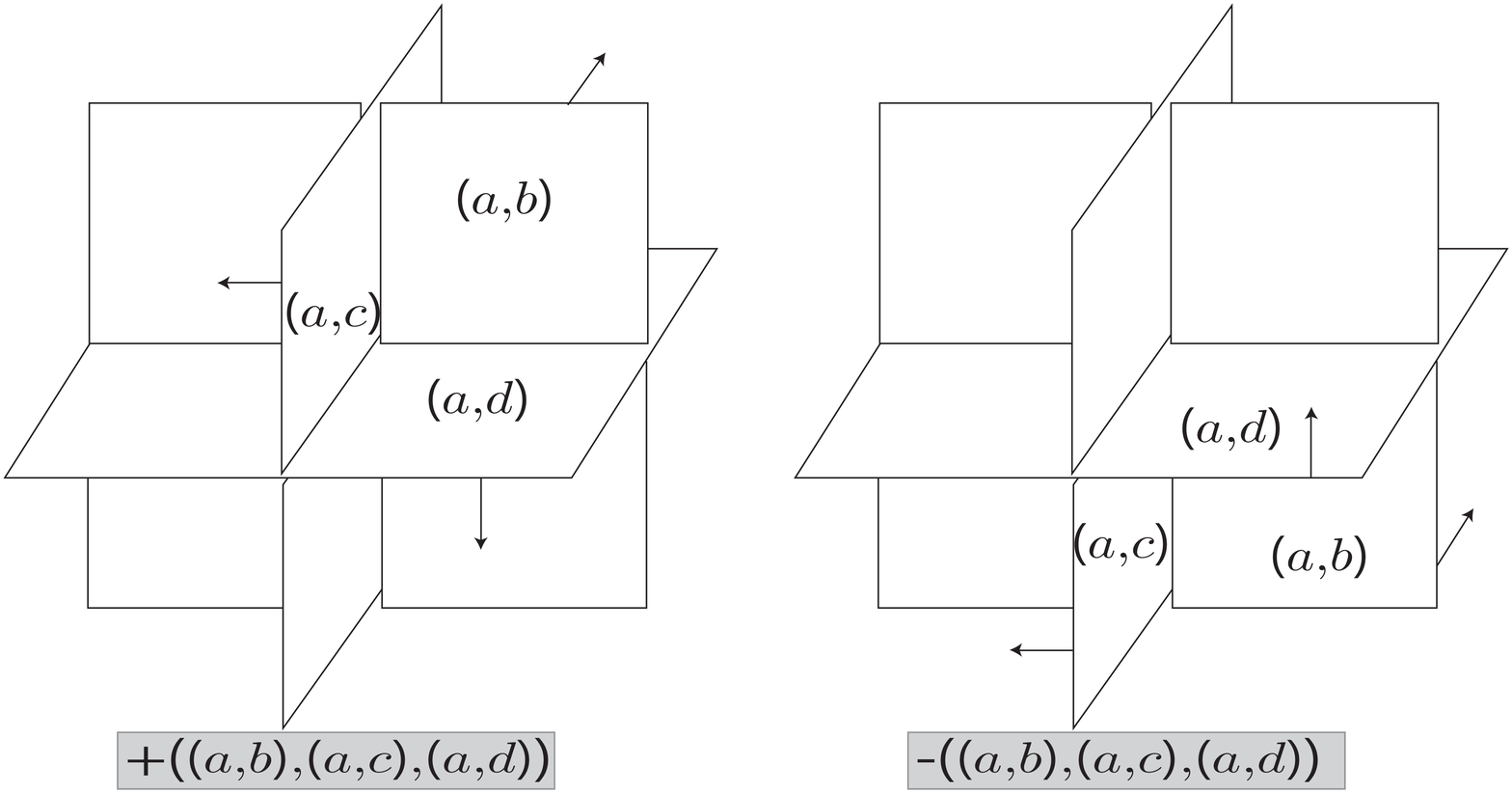}
    \caption{}
    \label{triplepoint}
  \end{center}
\end{figure}

Let $A$ be an abelian group. For a $3$-cocycle $\theta \in C^3_{\rm LB}(X; A)$, we define
\begin{align*}
&\mathcal{H}^{\rm LB}(D)=\big\{[W^{\rm LB}(D, C)] \in H^{\rm LB}_3(X) \ \big| \ C \in {\rm Col}_{X^2}^{\rm LB} (D) \big\},\mbox{ and }\\ 
&\Phi_{\theta}^{\rm LB}(D)=\big\{\theta\big(W^{\rm LB}(D, C)\big) \in A \ \big| \ C \in {\rm Col}_{X^2}^{\rm LB} (D) \big\}
\end{align*}
as multisets. Then we have the following theorem:
\begin{theorem}{\rm (\cite{NOO})}
$\mathcal{H}^{\rm LB}(D)$ and $\Phi_{\theta}^{\rm LB}(D)$ are invariants of $F$.
\end{theorem}


\section{Main results}\label{sec:main}

\subsection{Corresponding tribrackets and local biquandles}
\begin{theorem}\label{thm:1}
Given a shadow biquandle $(B, X, \uline{*}, \oline{*}, *)$ such that $X$ is  strongly connected, we have a horizontal-tribracket $[\,]: X^3\to X$ defined by 
\[
[x,y,z] = y *\big( (x \darrow  z) \oline{*} (x \darrow  y) \big)= z * \big( (x \darrow  y) \uline{*} (x \darrow  z) \big).
\]
\end{theorem}
\begin{proof}
We first show the second equality. It holds since
 \[
\begin{array}{ll}
y *\big( (x \darrow  z) \oline{*} (x \darrow  y) \big)&\overeq{Lem.}{lem:1} \big( x*(x \darrow  y) \big)*\big( (x \darrow  z) \oline{*} (x \darrow  y) \big)\\[3pt]
&\overset{{\rm Def.}\ref{def:B-set}}{=}\big( x* (x \darrow  z)  \big)*\big( (x \darrow  y) \uline{*} (x \darrow  z) \big)\\[3pt]
&\overset{{\rm Lem.}\ref{lem:1}}{=}z*\big( (x \darrow  y) \uline{*} (x \darrow  z) \big).
\end{array}
\]

Next we show the first equality by checking the horizontal-tribracket axioms one by one.

($\mathcal{H}1$)-(i) Suppose that  $x,y,w\in X$ are given.
Let $z= x*a$, where $a=(y\darrow {*} w) \oline{*}^{-1} (x \darrow  y)$. 
We then have 
\[
[x,y,z] =  y *\big( (x \darrow  (x*a)) \oline{*} (x \darrow  y) \big) \overeq{Lem.}{lem:1} y* \big(a \oline{*} (x \darrow  y) \big)=y*  (y \darrow  w) \overeq{Lem.}{lem:1} w.
\]
The uniqueness of the above $z$ holds as follows: 
Assume that $[x,y,z] =w = [x,y,z']$ for some $z,z'\in X$.
We then have 
\[
y *\big( (x\darrow  z) \oline{*} (x \darrow  y) \big)=w=y *\big( (x\darrow  z') \oline{*} (x \darrow  y) \big).
\]
Hence we have 
\[
x \darrow  z   = (y \darrow w )\oline{*}^{-1} (x \darrow  y) = x \darrow  z'.
\]
Then we have 
\[
 z \overeq{Lem.}{lem:1} x*  (x \darrow  z) = x*  (x \darrow  z')   \overeq{Lem.}{lem:1} z'.
\]


($\mathcal{H}1$)-(ii)
Suppose that  $x,z,w\in X$ are given.
Let $y= x*a$, where $a=(z\darrow  w) \uline{*}^{-1} (x \darrow  z)$. 
We then have 
\[
[x,y,z] =  z * \big( (x \darrow  (x*a)) \uline{*} (x \darrow  z) \big)
 \overeq{Lem.}{lem:1} z * \big( a \uline{*} (x \darrow  z) \big)
= z *  (z \darrow  w)  \overeq{Lem.}{lem:1} w .
\]
The uniqueness of the above $y$ holds as follows: 
Assume that $[x,y,z] =w = [x,y',z]$ for some $y,y'\in X$.
We then have 
\[
z * \big( (x \darrow  y) \uline{*} (x \darrow  z) \big)=w=z * \big( (x \darrow  y') \uline{*} (x \darrow  z) \big).
\]
Hence we have 
\[
x \darrow  y   = (z \darrow w )\uline{*}^{-1} (x \darrow  z) = x \darrow  y'.
\]
Therefore we have 
\[
 y \overeq{Lem.}{lem:1} x*  (x \darrow  y) = x*  (x \darrow  y')   \overeq{Lem.}{lem:1} y'.
\]


($\mathcal{H}1$)-(iii)
Suppose that  $y,z, w \in X$ are given. 
Let $x=y*^{-1} a = z *^{-1} b$, where $(a,b) \in B^2$ such that  $S(a, b) =(b\oline{*}a, a\uline{*} b) = ((y\darrow  w), (z\darrow  w))$ for the bijection $S: B^2 \to B^2$ in  Definition~\ref{def:biquandle}, and where it holds that 
\[
\begin{array}{ll}
y*^{-1} a &\overset{ {\rm Lem.}\ref{lem:1}}{=} \big( w*^{-1} (y\darrow  w ) \big) *^{-1} a 
\\[3pt]
&\ \ \ =\ \ \ \big(w*^{-1}(b\oline{*} a) \big) *^{-1} a\\[3pt]
&\overset{{\rm Lem.}\ref{lem:0}}{=}\big(w*^{-1} (a\uline{*} b) \big) *^{-1} b\\[3pt]
&\ \ \ =\ \ \ \big(w*^{-1} (z\darrow  w )\big) *^{-1} b\\[3pt]
& \overset{{\rm Lem.}\ref{lem:1}}{=} z*^{-1} b .
\end{array}
\]
We then have 
\[
\begin{array}{l}
[x,y,z] = y *\big( ((z*^{-1} b) \darrow  z) \oline{*} ((y*^{-1} a) \darrow  y) \big) 
\overeq{Lem.}{lem:1} 
y *( b \oline{*} a)  =y *( y\darrow w) =w. 
\end{array}
\]
The uniqueness of the above $x$ holds as follows: 
Assume that $[x,y,z] =w = [x',y,z]$ for some $x,x'\in X$.
We then have 
\[
y *\big( (x \darrow  z) \oline{*} (x \darrow  y) \big)=w=y *\big( ({x'}\darrow  z) \oline{*} ({x'}\darrow  y) \big)
\]
and 
\[
z * \big( (x \darrow  y) \uline{*} (x \darrow  z) \big)=w=z * \big( ({x'} \darrow  y) \uline{*} ({x'} \darrow  z) \big).
\]
Hence we have 
\[
 (x \darrow  z) \oline{*} (x \darrow  y) =y \darrow  w= ({x'} \darrow  z) \oline{*} ({x'} \darrow  y) 
\]
and 
\[
 (x \darrow  y) \uline{*} (x \darrow  z) =z \darrow  w= ({x'} \darrow  y) \uline{*} ({x' }\darrow  z).
\]
Since there exists a unique element $(c,d) \in B^2$ such that $S(c,d) = (d\oline{*} c, c\uline{*}d) = (y \darrow  w , z \darrow  w)$ by Definition~\ref{def:biquandle}, we have 
\[
x \darrow  y = c={x' }\darrow  y \mbox{  \ \ \ (and $x \darrow  z = d= {x'} \darrow  z$).}
\]
Therefore we have 
\[
x\overeq{Lem.}{lem:1} y*^{-1} (x \darrow  y) = y*^{-1} ({x' }\darrow  y) \overeq{Lem.}{lem:1} x'.
\]


($\mathcal{H}2$) For $x,y,z,w \in X$, we have 
\[
\begin{array}{lcl}
&& \hspace{-1.2cm}[y,[x,y,z],[x,y,w]] \\
&=& [x,y,z] * \Big( \big( y \darrow  [x,y,w] \big) \oline{*} \big( y \darrow  [x,y,z] \big) \Big)\\[4pt]
&=&  [x,y,z] * \Big( \big( y \darrow  \big(y*( (x\darrow  w) \oline{*} (x\darrow  y))\big) \big) \oline{*} \big( y \darrow  \big(y*( (x\darrow  z) \oline{*} (x\darrow  y))\big) \big) \Big)\\[4pt]
&\overeq{Lem.}{lem:1}& [x,y,z] * \Big( \big(  (x\darrow  w) \oline{*} (x\darrow  y) \big) \oline{*} \big( (x\darrow  z) \oline{*} (x\darrow  y) \big) \Big)\\[4pt]
&\overeq{Def.}{def:biquandle}& [x,y,z] * \Big( \big(  (x\darrow  w) \oline{*} (x\darrow  z) \big) \oline{*} \big( (x\darrow  y) \uline{*} (x\darrow  z) \big) \Big)\\[4pt]
&\overeq{Lem.}{lem:1}& [x,y,z] * \Big( \big( z \darrow  \big( z*  (  (x\darrow  w) \oline{*} (x\darrow  z) ) \big) \big)  \oline{*}  \big( z \darrow  \big( z* ( (x\darrow  y) \uline{*} (x\darrow  z) )  \big) \big) \Big)\\[4pt]
&\overeq{Lem.}{lem:1}& [x,y,z] * \Big( \big( z \darrow  [x,z,w] \big)  \oline{*}  \big( z \darrow  [x,y,z] \big) \Big)\\[4pt]
&=& [z,[x,y,z],[x,z,w]],
\end{array} 
\]
and
\[
\begin{array}{lcl}
&& \hspace{-1.2cm} [w,[x,y,w],[x,z,w]] \\
&=& [x,z,w] * \Big( \big( w \darrow  [x,y,w] \big) \uline{*} \big( w \darrow  [x,z,w] \big) \Big)\\[4pt]
&=&  [x,z,w]  * \Big( \big( w \darrow  \big(w*( (x\darrow  y) \uline{*} (x\darrow  w))\big) \big) \uline{*} \big( w \darrow  \big(w*( (x\darrow  z) \uline{*} (x\darrow  w))\big) \big) \Big)\\[4pt]
&\overeq{Lem.}{lem:1}&  [x,z,w]  * \Big( \big(  (x\darrow  y) \uline{*} (x\darrow  w) \big) \uline{*} \big( (x\darrow  z) \uline{*} (x\darrow  w) \big) \Big)\\[4pt]
&\overeq{Def.}{def:biquandle}& [x,z,w]  * \Big( \big(  (x\darrow  y) \uline{*} (x\darrow  z) \big) \uline{*} \big( (x\darrow  w) \oline{*} (x\darrow  z) \big) \Big)\\[4pt]
&\overeq{Lem.}{lem:1}& [x,z,w]  * \Big( \big( z \darrow  \big( z*  (  (x\darrow  y) \uline{*} (x\darrow  z)  ) \big) \big)  \uline{*}  \big( z \darrow  \big( z* ( (x\darrow  w) \oline{*} (x\darrow  z) )  \big) \big) \Big)\\[4pt]
&\overeq{Lem.}{lem:1}& [x,z,w]  * \Big( \big( z \darrow  [x,y,z] \big)  \uline{*}  \big( z \darrow  [x,z,w] \big) \Big)\\[4pt]
&=& [z,[x,y,z],[x,z,w]].
\end{array} 
\]
This completes the proof.
\end{proof}
\begin{definition}
For a shadow biquandle $(B,X)$ such that $X$ is strongly connected,  
we call the horizontal-tribacket $[\,]$ given in Theorem~\ref{thm:1} the {\it corresponding horizontal-tribacket} of $(B,X)$.
We call the local biquandle $(X, \{\oline{\star}\}, \{\uline{\star}\})$ associated with the corresponding horizontal-tribacket $[\,]$ of $(B,X)$ the {\it corresponding local biquandle} of $(B,X)$. 
\end{definition}

\subsection{Correspondence between (co)homology groups}\label{subsec:Correspondence between (co)homology groups}
 Let $(B, X,  \uline{*}, \oline{*}, *)$ be a shadow biquandle such that $X$ is strongly connected.
Let $[\,] :X^3\to X$ be the corresponding horizontal-tribracket of $(B, X)$, that is, it is defined by 
\[
[x,y,z] = y *\big( (x \darrow  z) \oline{*} (x \darrow  y) \big)= z * \big( (x \darrow  y) \uline{*} (x \darrow  z) \big).
\] 
Let $(X, \{ \uline{\star}\} , \{ \oline{\star}\})$ be the corresponding local biquandle of $(B,X)$, that is, it is the local biquandle associated with the above $(X,[\, ])$.
Define a homomorphism $\mu_n:  C_n^{\rm SB} (B, X)  \to C_n^{\rm LB} (X)$ by 
\[
\mu_n\big((x,a_1, \ldots , a_n)\big) = \big( (x, x* a_1 ), \ldots , (x, x* a_n )\big)
\]
if  $n\geq1$, and $\mu_n=0$ otherwise.

\begin{lemma}\label{prop:2}
$\mu_n$ is a bijective chain map. 
\end{lemma}
\begin{proof}
It is sufficient to consider the cases that $n\geq 1$.

We first show that $\mu_n$ is well-defined. 
For $(x,a_1, \ldots , a_n)\in X \times B^n$, suppose that $a_i = a_{i+1}$ for some $i\in \{1, \ldots, n-1\}$. 
We then have 
\[
\begin{array}{ll}
\mu_n\big((x,a_1, \ldots , a_n)\big) &= \big( (x, x* a_1 ), \ldots, (x, x* a_i ) , (x, x* a_{i+1} ), \ldots  , (x, x* a_n )\big)\\[3pt]
&= \big( (x, x* a_1 ), \ldots, (x, x* a_i ) , (x, x* a_{i} ), \ldots  , (x, x* a_n )\big).
\end{array}
\]
This implies $\mu_n$ is well-defined since $\mu_n\big(D_n^{\rm sb}(B, X) \big)\subset D_n^{\rm lb}(X)$ when we regard $\mu_n$ as a homomorphism from $C_n^{\rm sb}(B, X)$ to $C_n^{\rm lb}(X)$.

Next we show that $\mu_n$ is bijective. 
Define a homomorphism $\eta_n:  C_n^{\rm LB} (X) \to C_n^{\rm SB} (B, X) $ by 
\[
\eta_n\Big(\big( (x, y_1), \ldots , (x, y_n) \big)\Big) = (x, x\darrow y_1, \ldots , x\darrow y_n)
\]
if $n\geq 1$, and $\eta_n=0$ otherwise.
Then if $y_i=y_{i+1}$ for some $i\in \{1,\ldots , n-1\}$, 
\[
\begin{array}{l}
\eta_n\Big(\big( (x, y_1), \ldots, (x, y_i), (x, y_{i+1}=y_{i}), \ldots  , (x, y_n) \big)\Big) \\[3pt]
= (x, x\darrow y_1, \ldots ,x\darrow y_i, x\darrow y_{i+1} = x\darrow y_{i} , \ldots,   x\darrow y_n).
\end{array}
\] 
This implies that $\eta_n$ is well-defined since $\eta_n\big( D_n^{\rm lb} (X)\big) \subset D_{n}^{\rm sb}(B,X)$ when we regard $\eta_n$ as a homomorphism from $C_n^{\rm lb} (X)$ to $C_{n}^{\rm sb}(B,X)$.
For $n\geq 1$,  we have 
\[
\begin{array}{lcl}
\eta_n\circ \mu_n\big((x,a_1, \ldots , a_n)\big) &=&\eta_n \Big(\big( (x, x* a_1 ), \ldots  , (x, x* a_n )\big)\Big)\\[3pt]
&=&\big( x, x \darrow (x*a_1) , \ldots ,  x \darrow (x*a_n)  \big)\\[2pt]
&\overeq{Lem.}{lem:1}& \big( x, a_1, \ldots , a_n \big), 
\end{array}
\]
and 
\[
\begin{array}{lcl}
\mu_n \circ \eta_n\Big(\big( (x, y_1), \ldots , (x, y_n) \big)\Big) &=&\mu_n \Big((x, x\darrow y_1, \ldots , x\darrow y_n)\Big)\\[3pt]
&=&\Big( \big(x, x * (x\darrow y_1) \big) , \ldots ,  \big(x, x * (x\darrow y_n) \big) \Big)\\[2pt]
&\overeq{Lem.}{lem:1}&\big( (x, y_1), \ldots , (x, y_n) \big).
\end{array}
\]
Hence $\eta_n$ is the inverse map of $\mu_n$, and thus,  $\mu_n$ is bijective.

Lastly, we show that $\mu_n$ is a chain map. 
We have 
\begin{align}
&\mu_{n-1} \circ \partial_n^{\rm SB}\big( (x, a_1, \ldots , a_n )\big) \notag \\
&= \mu_{n-1} \Big( \sum_{i=1}^{n} (-1)^i \big( x, a_1, \ldots , a_{i-1}, a_{i+1}, \ldots , a_n \big) \notag \\
& \hspace{0.4cm}+\sum_{i=1}^{n} (-1)^{i+1} \big( x * a_i, a_1 \uline{*} a_i, \ldots , a_{i-1} \uline{*} a_i , a_{i+1}\oline{*} a_i, \ldots , a_n \oline{*} a_i \big) \Big)  \notag \\
&=\sum_{i=1}^{n} (-1)^i \big( (x, x*a_1), \ldots , (x, x*a_{i-1}), (x, x*a_{i+1}), \ldots , (x, x*a_n) \big) \\
& \hspace{0.4cm}+\sum_{i=1}^{n} (-1)^{i+1} \Big( \big( x * a_i, (x * a_i) *(a_1 \uline{*} a_i) \big), \ldots , \big( x * a_i, (x * a_i) *(a_{i-1} \uline{*} a_i) \big),  \notag \\ &\hspace{2.3cm}  \big( x * a_i, (x * a_i) *(a_i \oline{*} a_{i+1}) \big), \ldots , 
\big( x * a_i, (x * a_i) *(a_i \oline{*} a_{n}) \big)  \Big)
\end{align}
and 
\begin{align}
& \partial_n^{\rm LB} \circ \mu_n \big( (x, a_1, \ldots , a_n )\big) \notag \\
&= \partial_n^{\rm LB} \Big(\big( (x, x* a_1 ), \ldots , (x, x* a_n )\big) \Big) \notag\\
&=\sum_{i=1}^{n} (-1)^i \big( (x, x*a_1), \ldots , (x, x*a_{i-1}), (x, x*a_{i+1}), \ldots , (x, x*a_n) \big) \\
& \hspace{0.4cm}+\sum_{i=1}^{n} (-1)^{i+1} \big((x, x* a_1 ) \uline{\star} (x, x* a_{i} ), \ldots ,   (x, x* a_{i-1} ) \uline{\star} (x, x* a_{i} ), \notag\\
 &\hspace{2.7cm}  (x, x* a_{i+1} ) \oline{\star} (x, x* a_{i} ), \ldots , (x, x* a_{n} ) \oline{\star} (x, x* a_{i} ) \big).
\end{align}
We can easily see that the terms (1) coincide with the terms (3).
The terms (2) coincide with the terms (4) because 
for $1\leq j < i$, it holds that
\[
\begin{array}{lcl}
(x, x* a_j ) \uline{\star} (x, x* a_{i} ) &=& \big(x* a_{i}, [x, x* a_j , x* a_{i}]\big) \\[3pt]
&=& \Big(x* a_{i},   (x* a_{i})* \big((x \darrow (x* a_j) ) \uline{*} (x \darrow (x* a_{i}) )\big) \Big) \\[2pt]
&\overeq{Lem.}{lem:1}& \big(x* a_{i},   (x* a_{i})* (a_j \uline{*} a_{i} )\big),
\end{array}
\]
and for $i<j \leq n$, it holds that 
\[
\begin{array}{lcl}
 (x, x* a_{j} ) \oline{\star} (x, x* a_{i} ) &=& \big(x* a_{i}, [x,  x* a_{i}, x* a_j]\big) \\[3pt]
&=& \Big(x* a_{i},   (x* a_{i})* \big((x \darrow (x* a_j) ) \oline{*} (x \darrow (x* a_{i}) )\big) \Big) \\[2pt]
&\overeq{Lem.}{lem:1}& \big(x* a_{i},   (x* a_{i})* (a_j \oline{*} a_{i} )\big).
\end{array}
\]
Therefore we have 
\[
\mu_{n-1} \circ \partial_n^{\rm SB} = \partial_n^{\rm LB} \circ \mu_n, 
\]
and thus, $\mu_n$ is a chain map.

This completes the proof.
\end{proof}

The bijective chain map $\mu_n$ induces  an isomorphism $\mu_n^*:  H_n^{\rm SB} (B, X)  \to H_n^{\rm LB} (X)$ defined by 
\[
\mu_n^*\Big(\big[(x,a_1, \ldots , a_n)\big]\Big) = \Big[\mu_n \big( (x,a_1, \ldots , a_n)\big)\Big]
\]
if  $n\geq1$, and $\mu_n^*=0$ otherwise.

Moreover, for an abelian group $A$, 
 the bijective chain map $\mu_n$ induces  the bijective chain map $\mu_n \otimes {\rm id}:   C_n^{\rm SB} (B, X; A)  \to C_n^{\rm LB} (X; A)$, and hence, we have an isomorphism 
$(\mu_n \otimes {\rm id} )^* :   H_n^{\rm SB} (B, X; A)  \to H_n^{\rm LB} (X; A)$. 
The bijective cochain map $\mu_n$ induces  the bijective cochain map $\mu^n:   C^n_{\rm LB} (X; A) \to C^n_{\rm SB} (B, X; A)$ defined by $\mu^n (f) = f \circ \mu_n$, and hence, we have an isomorphism 
$\mu^n_*: H^n_{\rm LB} (X; A) \to H^n_{\rm SB} (B, X; A)$. 
Thus we have the following theorem:
\begin{theorem}\label{thm:2}
Let $(B, X)$ be a shadow biquandle such that $X$ is strongly connected, and  $(X, \{ \uline{\star}\} , \{ \oline{\star}\})$ be the corresponding local biquandle of $(B,X)$. Let $A$ be an abelian group. 
Then for any $n\in \mathbb Z$, we have 
\[
H_n^{\rm SB} (B, X; A)  \cong H_n^{\rm LB} (X; A) \mbox{\   and  \ } 
H^n_{\rm SB} (B, X; A)  \cong H^n_{\rm LB} (X; A).
\]
\end{theorem}

\subsection{Correspondence between cocycle invariants of links}
 Let $(B, X,  \uline{*}, \oline{*}, *)$ be a shadow biquandle such that $X$ is strongly connected.
Let $[\,] :X^3\to X$ be the corresponding horizontal-tribracket of $(B, X)$, that is, it is defined by 
\[
[x,y,z] = y *\big( (x \darrow  z) \oline{*} (x \darrow  y) \big)= z * \big( (x \darrow  y) \uline{*} (x \darrow  z) \big).
\] 
Let $(X, \{ \uline{\star}\} , \{ \oline{\star}\})$ be the corresponding local biquandle of $(B,X)$, that is, it is the local biquandle associated with the above $(X,[\, ])$.

Let $D$ be a connected  diagram of a link $L$. 
\begin{lemma}\label{prop:translation1}
There exists a bijection $ T: {\rm Col}_{(B,X)}^{\rm SB} (D) \to {\rm Col}_{X^2}^{\rm LB} (D)$.
\end{lemma}
\begin{proof}
We set a map 
$T: {\rm Col}_{(B,X)}^{\rm SB} (D) \to {\rm Col}_{X^2}^{\rm LB} (D)$ as follows:
Let $C \in {\rm Col}_{(B,X)}^{\rm SB} (D)$. 
For a semi-arc $s$ whose normal vector points from a  region $r_1$ to a region $r_2$ as shown in the right of Figure~\ref{coloring6}, we assign $(x,y)$ to the semi-arc $s$, where $x=C(r_1)$ and $y=C(r_2)$, see also Figure~\ref{coloring7}.
Then the assignment determines an $X^2$-coloring $C'=T(C) \in {\rm Col}_{X^2}^{\rm LB} (D)$.
Indeed, since $$w=y*\big( b \oline{*} a \big)  =y*\big( (x\darrow z) \oline{*}(x\darrow y)\big)=[x,y,z]  
$$  
for a crossing of $(D,C)$ as shown in the left of Figure~\ref{coloring8}, 
for the same crossing of $(D,C')$ as shown in the right of Figure~\ref{coloring8}, the conditions of a local biquandle coloring in Definition~\ref{def:localbiquandlecolor} hold as follows:
\[
\begin{array}{l}
C'(u_1) \uline{\star} C'(o_1) = (x,y) \uline{\star} (x,z) = (z, [x,y,z])=(z,w) = C'(u_2), \mbox{ and} \\[3pt]
C'(o_1) \oline{\star} C'(u_1) = (x,z) \oline{\star} (x,y) = (y, [x,y,z]) =(y,w) = C'(o_2).
\end{array}
\]

The inverse map $T^{-1} : {\rm Col}_{X^2}^{\rm LB} (D) \to {\rm Col}_{(B,X)}^{\rm SB} (D)$ is defined as follows: 
Let $C'\in {\rm Col}_{X^2}^{\rm LB} (D)$. 
For a semi-arc $s$  whose normal vector points from a  region $r_1$ to a region $r_2$ as shown in the right of Figure~\ref{coloring6}, we assign $x$ to the region $r_1$, $y$ to the region $r_2$, and $x\darrow y$ to the semi-arc $s$, where $C'(s)=(x,y)$, see also Figure~\ref{coloring7}.
Then the assignment determines a $(B,X)$-coloring $C\in {\rm Col}_{(B,X)}^{\rm SB} (D)$. 
Indeed, since $w=[x,y,z]$ for a crossing of $(D,C')$ as shown in the right of Figure~\ref{coloring8}, 
for the same crossing of $(D,C)$ as shown in the left of Figure~\ref{coloring8}, the conditions of a shadow biquandle coloring in Definition~\ref{def:coloringshadow} hold as follows:
\[
\begin{array}{l}
C(u_1) \uline{*} C(o_1) = (x\darrow y) \uline{*} (x \darrow  z)\overeq{Lem.}{lem:1}z\darrow \big(z* \big( (x\darrow y) \uline{*} (x \darrow  z)\big) \big) \\[3pt]
= z\darrow  [x,y,z]=  z\darrow  w= C(u_2),  \mbox{ and }\\[3pt]
C(o_1) \oline{*} C(u_1) = (x\darrow z) \oline{*} (x\darrow y)\overeq{Lem.}{lem:1}y\darrow  \big(y*\big((x\darrow z) \oline{*} (x\darrow y)\big)\big)\\[3pt] = y\darrow  [x,y,z] =  y\darrow  w = C(o_2).
\end{array}
\]
For a semi-arc of $(D,C)$ as shown in the right of Figure~\ref{coloring6}, 
\[
C(r_1) * C(s) = x * (x\darrow y) \overeq{Lem.}{lem:1} y = C(r_2),
\]
and thus, the condition of a shadow biquandle coloring around each semi-arc also holds.

Therefore $T$ is bijective. 

\begin{figure}[ht]
  \begin{center}
    \includegraphics[clip,width=5.0cm]{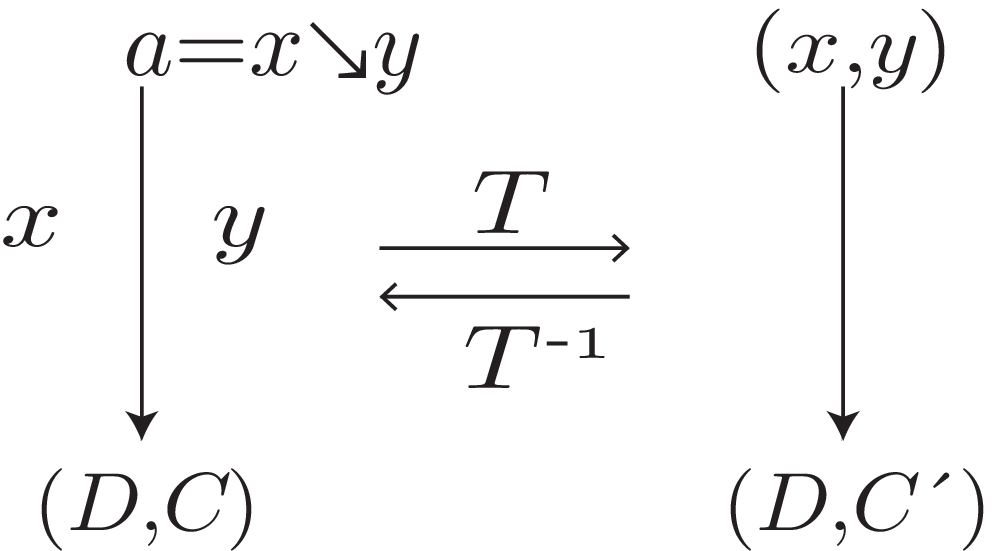}
    \caption{}
    \label{coloring7}
  \end{center}
\end{figure}
\begin{figure}[ht]
  \begin{center}
    \includegraphics[clip,width=8.0cm]{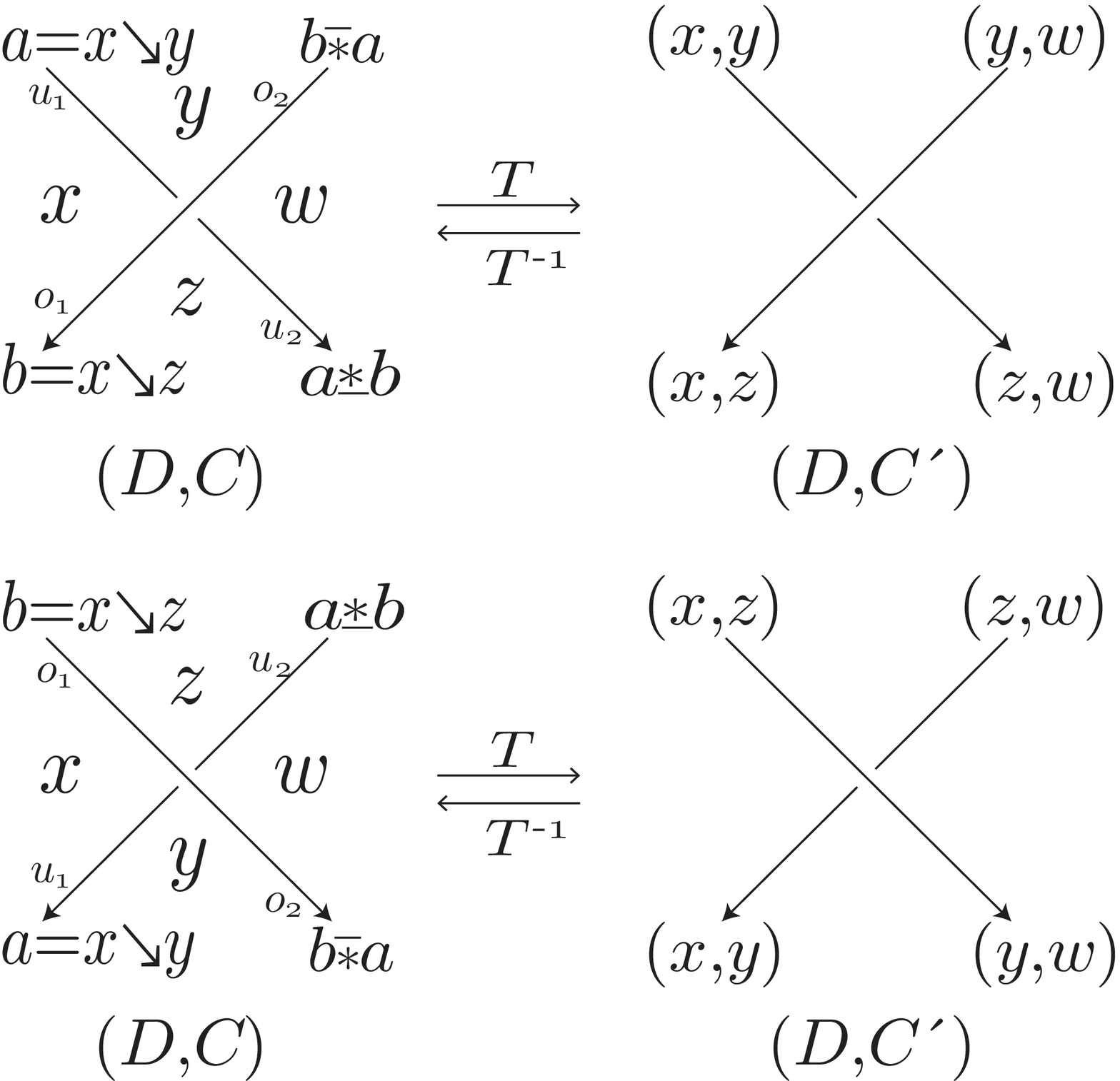}
    \caption{}
    \label{coloring8}
  \end{center}
\end{figure}
\end{proof}

\noindent We continue to use the bijection $T: {\rm Col}_{(B,X)}^{\rm SB}(D) \to {\rm Col}_{X^2}^{\rm LB}(D)$. 

Let  $\mu_2:  C_2^{\rm SB} (B, X)  \to C_2^{\rm LB} (X)$ be the bijective chain map defined in Subsection~\ref{subsec:Correspondence between (co)homology groups}, that is, it is defined by 
\[
\mu_2\big((x, a_1, x_2)\big) = \big( (x, x*a_1), (x, x*a_2)\big).
\]
We note that the inverse map $\mu^{-1}_2 (=\eta_2 \mbox{ in Subsection~\ref{subsec:Correspondence between (co)homology groups}})$ of $\mu_2$ is defined by 
\[
\mu_2^{-1} \Big( \big((x,y), (x,z)\big) \Big) = (x, x \darrow y, x \darrow z).
\]
Let $C\in {\rm Col}_{(B,X)}^{\rm SB}(D)$ and $C'\in {\rm Col}_{X^2}^{\rm LB}(D)$ such that $T(C)=C'$.
At a crossing  $\chi$ of $D$ as depicted in Figure~\ref{coloring6}, we have 
\begin{align*}
w^{\rm LB}(D, C'; \chi) 
&={\rm sign}(\chi) \big((x,y), (x,z)\big)\\
&=\mu_2 \Big({\rm sign}(\chi)\big(x, x \darrow y, x \darrow z \big)\Big)\\
&=\mu_2\big(w^{\rm SB}(D, C; \chi)\big),
\end{align*}
where $C'(u_1) = (x,y)$ and $C'(o_1)=(x,z)$, see also Figure~\ref{coloring8}.
This implies that $W^{\rm LB}(D, C') =\mu_2\big(W^{\rm SB}(D, C)\big)$. Thus we have 
\begin{align*}
\mathcal{H}^{\rm LB}(D)&=\Big\{\big[W^{\rm LB}(D, C')\big]  ~\Big|~ \ C' \in {\rm Col}_{X^2}^{\rm LB} (D) \Big\}\\
&=\Big\{\mu_2^{\ast}\big(\big[W^{\rm SB}(D, C)\big]\big) ~\Big|~ \ C \in {\rm Col}_{(B,X)}^{\rm SB} (D) \Big\}\\
&=\mu_2^{\ast}\big(\mathcal{H}^{\rm SB}(D)\big).
\end{align*}
We note that since $\mu_2^{*}$ is an isomorphism, $\mathcal{H}^{\rm SB}(D)=(\mu_2^{\ast})^{-1}\big( \mathcal{H}^{\rm LB}(D) \big)$ holds.
This implies that as link invariants, $\mathcal{H}^{\rm SB}(L)$ and $\mathcal{H}^{\rm LB}(L)$ are the same.

Let $A$ be an abelian group.
Let   $\theta \in Z^2_{\rm SB}(B, X;A)$ and $\theta' \in Z^2_{\rm LB}(X;A)$ such that $\theta = \theta' \circ \mu_2$.
We then have  
\[
\theta \big(w^{\rm SB}(D, C; \chi)\big)= \theta' \circ \mu_2 \big(w^{\rm SB}(D, C; \chi)\big)
=\theta' \big(w^{\rm LB}(D, C'; \chi)\big)
\]
for each crossing $\chi$, and thus,  $\theta \big(W^{\rm SB}(D, C)\big)=\theta'(W^{\rm LB}(D, C'))$ holds. 
This implies that $\Phi_{\theta}^{\rm SB}(L) =\Phi_{\theta'}^{\rm LB}(L)$. 

As a consequence, we have the following theorem:
\begin{theorem}\label{thm:3}
Let $L$ be a link.
Let $(B, X)$ be a shadow biquandle such that $X$ is strongly connected, and  $(X, \{ \uline{\star}\} , \{ \oline{\star}\})$ be the corresponding local biquandle of $(B,X)$. 
Then we have 
\[
\mathcal{H}^{\rm LB}(L)=\mu_2^* \big( \mathcal{H}^{\rm SB}(L)\big) \mbox{ \ \ and \ \  } 
(\mu_2^*)^{-1}\big( \mathcal{H}^{\rm LB}(L)\big) =\mathcal{H}^{\rm SB}(L).
\]
Moreover for an abelian group $A$, let $\theta \in Z^2_{\rm SB}(B, X;A)$ and $\theta' \in Z^2_{\rm LB}(X;A)$ such that $\theta = \theta' \circ \mu_2$.
Then we have 
$$\Phi_{\theta}^{\rm SB}(L) =\Phi_{\theta'}^{\rm LB}(L).$$
\end{theorem}



\subsection{Correspondence between cocycle invariants of surface-links}
 Let $(B, X,  \uline{*}, \oline{*}, *)$ be a shadow biquandle such that $X$ is strongly connected.
Let $[\,] :X^3\to X$ be the corresponding horizontal-tribracket of $(B, X)$, that is, it is defined by 
\[
[x,y,z] = y *\big( (x \darrow  z) \oline{*} (x \darrow  y) \big)= z * \big( (x \darrow  y) \uline{*} (x \darrow  z) \big).
\] 
Let $(X, \{ \uline{\star}\} , \{ \oline{\star}\})$ be the corresponding local biquandle of $(B,X)$, that is, it is the local biquandle associated with the above $(X,[\, ])$.

Let $D$ be a connected diagram of a surface-link $F$. 
\begin{lemma}
There exists a bijection $ T: {\rm Col}_{(B,X)}^{\rm SB} (D) \to {\rm Col}_{X^2}^{\rm LB} (D)$.
\end{lemma}
\begin{proof}
Here, we show only how to construct a bijection $T$, and the details are left to the reader, refer to the proof of Lemma~\ref{prop:translation1}. 

We set a map 
$T: {\rm Col}_{(B,X)}^{\rm SB} (D) \to {\rm Col}_{X^2}^{\rm LB} (D)$ as follows:
Let $C \in {\rm Col}_{(B,X)}^{\rm SB} (D)$. 
For a semi-sheet $s$ whose normal vector points from a  region $r_1$ to a region $r_2$ as shown in the right of Figure~\ref{doublepointshadow}, we assign $(x,y)$ to the semi-sheet $s$, where $x=C(r_1)$ and $y=C(r_2)$, see also Figure~\ref{doublepoint2}.
Then the assignment determines an $X^2$-coloring $C'=T(C) \in {\rm Col}_{X^2}^{\rm LB} (D)$.

The inverse map $T^{-1} : {\rm Col}_{X^2}^{\rm LB} (D) \to {\rm Col}_{(B,X)}^{\rm SB} (D)$ is defined as follows: 
Let $C'\in {\rm Col}_{X^2}^{\rm LB} (D)$. 
For a semi-sheet $s$ whose normal vector points from a  region $r_1$ to a region $r_2$ as shown in the right of Figure~\ref{doublepointshadow}, we assign $x$ to the region $r_1$, $y$ to the region $r_2$, and $x\darrow y$ to the semi-sheet $s$, where $C'(s)=(x,y)$, see also Figure~\ref{doublepoint2}.
Then the assignment determines a $(B,X)$-coloring $C\in {\rm Col}_{(B,X)}^{\rm SB} (D)$. 

Therefore $T$ is bijective. 

\begin{figure}[ht]
  \begin{center}
    \includegraphics[clip,width=6.0cm]{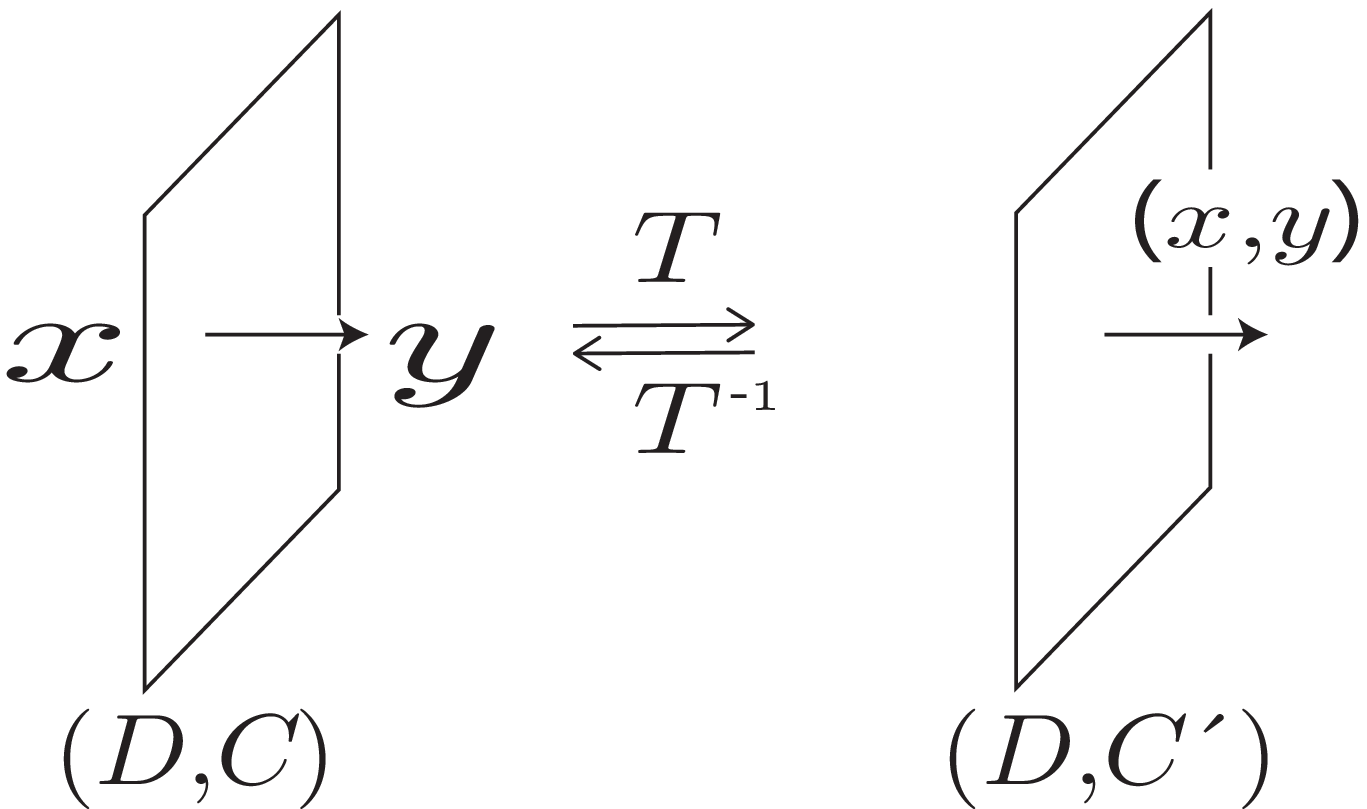}
    \caption{}
    \label{doublepoint2}
  \end{center}
\end{figure}

\end{proof}

\noindent We continue to use the bijection $T: {\rm Col}_{(B,X)}^{\rm SB}(D) \to {\rm Col}_{X^2}^{\rm LB}(D)$. 

Let  $\mu_3:  C_3^{\rm SB} (B, X)  \to C_3^{\rm LB} (X)$ be the bijective chain map defined in Subsection~\ref{subsec:Correspondence between (co)homology groups}, that is, it is defined by 
\[
\mu_3\big((x, a_1, x_2, a_3)\big) = \big( (x, x*a_1), (x, x*a_2) , (x, x*a_3)\big).
\]
We note that the inverse map $\mu^{-1}_3 (=\eta_3 \mbox{ in Subsection~\ref{subsec:Correspondence between (co)homology groups}})$ of $\mu_3$ is defined by 
\[
\mu_3^{-1} \Big( \big((x,y), (x,z), (x,w)\big) \Big) = (x, x \darrow y, x \darrow z, x \darrow w).
\]
Let $C\in {\rm Col}_{(B,X)}^{\rm SB}(D)$ and $C'\in {\rm Col}_{X^2}^{\rm LB}(D)$ such that $T(C)=C'$.
At a triple point  $\tau$ of $D$ as depicted in Figure~\ref{triplepoint4}, we have 
\begin{align*}
w^{\rm LB}(D, C'; \tau) 
&={\rm sign}(\tau) \big((x,y), (x,z), (x,w)\big)\\
&=\mu_3 \Big({\rm sign}(\tau)\big(x, x \darrow y, x \darrow z, x \darrow w \big)\Big)\\
&=\mu_3\big(w^{\rm SB}(D, C; \tau)\big),
\end{align*}
where $C'(b_1) = (x,y)$, $C'(m_1)=(x,z)$ and $C'(t_1)=(x,w)$, see also Figure~\ref{triplepoint6}.
This implies that $W^{\rm LB}(D, C') =\mu_3\big(W^{\rm SB}(D, C)\big)$. Thus we have 
\begin{align*}
\mathcal{H}^{\rm LB}(D)&=\Big\{\big[W^{\rm LB}(D, C')\big]  ~\Big|~ \ C' \in {\rm Col}_{X^2}^{\rm LB} (D) \Big\}\\
&=\Big\{\mu_3^{\ast}\big(\big[W^{\rm SB}(D, C)\big]\big) ~\Big|~ \ C \in {\rm Col}_{(B,X)}^{\rm SB} (D) \Big\}\\
&=\mu_3^{\ast}\big(\mathcal{H}^{\rm SB}(D)\big).
\end{align*}
We note that since $\mu_3^{*}$ is an isomorphism, $\mathcal{H}^{\rm SB}(D)=(\mu_3^{\ast})^{-1}\big( \mathcal{H}^{\rm LB}(D) \big)$ holds.
This implies that as surface-link invariants, $\mathcal{H}^{\rm SB}(F)$ and $\mathcal{H}^{\rm LB}(F)$ are the same.
\begin{figure}[ht]
  \begin{center}
    \includegraphics[clip,width=9.0cm]{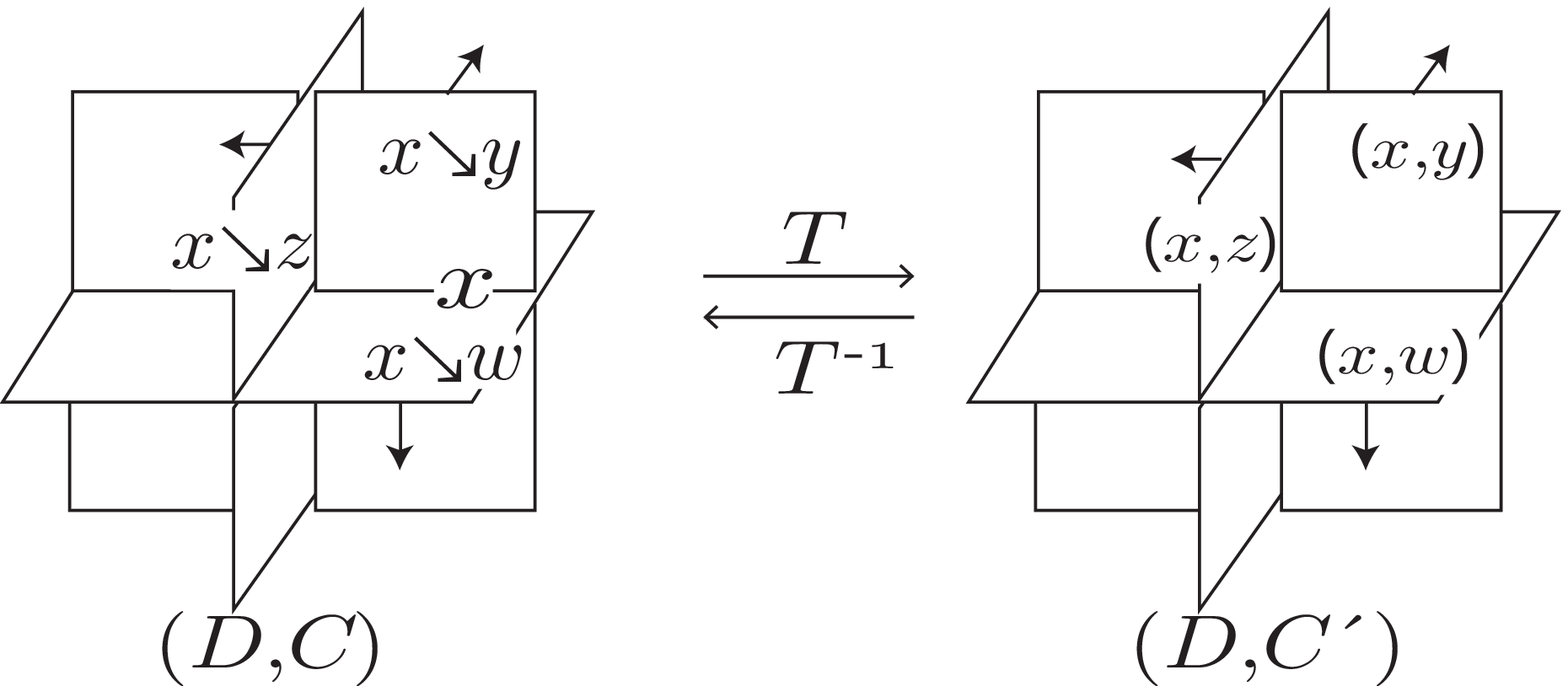}
    \caption{}
    \label{triplepoint6}
  \end{center}
\end{figure}

Let $A$ be an abelian group.
Let   $\theta \in Z^3_{\rm SB}(B, X;A)$ and $\theta' \in Z^3_{\rm LB}(X;A)$ such that $\theta = \theta' \circ \mu_3$.
We then have  
\[
\theta \big(w^{\rm SB}(D, C; \tau)\big)= \theta' \circ \mu_3 \big(w^{\rm SB}(D, C; \tau)\big)
=\theta' \big(w^{\rm LB}(D, C'; \tau)\big)
\]
for each triplepoint $\tau$, and thus,  $\theta \big(W^{\rm SB}(D, C)\big)=\theta'(W^{\rm LB}(D, C'))$ holds. 
This implies that $\Phi_{\theta}^{\rm SB}(F) =\Phi_{\theta'}^{\rm LB}(F)$. 

As a consequence, we have the following theorem:
\begin{theorem}\label{thm:4}
Let $F$ be a surface-link.
Let $(B, X)$ be a shadow biquandle such that $X$ is strongly connected, and  $(X, \{ \uline{\star}\} , \{ \oline{\star}\})$ be the corresponding local biquandle of $(B,X)$. 
Then we have 
\[
\mathcal{H}^{\rm LB}(F)=\mu_3^* \big( \mathcal{H}^{\rm SB}(F)\big) \mbox{ \ \ and \ \  } 
(\mu_3^*)^{-1}\big( \mathcal{H}^{\rm LB}(F)\big) =\mathcal{H}^{\rm SB}(F).
\]
Moreover for an abelian group $A$, let $\theta \in Z^3_{\rm SB}(B, X;A)$ and $\theta' \in Z^3_{\rm LB}(X;A)$ such that $\theta = \theta' \circ \mu_3$.
Then we have 
$$\Phi_{\theta}^{\rm SB}(F) =\Phi_{\theta'}^{\rm LB}(F).$$
\end{theorem}



\section{Remarks}

\subsection{Shadow biquandle theory and Niebrzydowski's theory}

In \cite{Niebrzydowski0, Niebrzydowski1, Niebrzydowski2}, region colorings of link diagrams by using algebraic structures called {\it knot-theoretic ternary quasigroups} were studied and used to define invariants of links and surface-links. Furthermore, Niebrzydowski in \cite{Niebrzydowski1, Niebrzydowski2} introduced a (co)homology theory of the algebraic structures, and defined a cocycle invariant for links and surface-links. 
In this subsection, we denote by $H_n^{\rm N} (X;A)$ and $H^n_N(X;A)$ the $n$th Niebrzydowski's homology group and cohomology group, respectively, for a given knot-theoretic horizontal-ternary-quasigroup $(X, [\,])$ and an abelian group $A$. 
Note that several versions of Niebrzydowski's (co)homology groups  were defined in \cite{Niebrzydowski1, Niebrzydowski2}, and in this subsection, his (co)homology groups mean the (co)homology groups reviewed in \cite{NOO}.
In addition, we denote by $\mathcal{H}^{\rm N}(L)$ and $\Phi^{\rm N}_\theta(L)$ the link invariants for a link $L$ using the homology group $H_1^{\rm N} (X)$ and a $1$-cocycle $\theta$  of his homology theory, respectively.
We denote by $\mathcal{H}^{\rm N}(F)$ and $\Phi^{\rm N}_\theta(F)$ the surface-link invariants for a surface-link $F$ using the homology group $H_2^{\rm N} (X)$ and a $2$-cocycle $\theta$ of his homology theory, respectively, see  \cite{NOO} for details.

In \cite{NOO}, we introduced local biquandle theory to show that the Niebrzydowski's (co)homology theory can be
interpreted as local biquandle (co)homology theory. On other words, the Niebrzydowski's (co)homology theory can be interpreted similarly as biquandle (co)homology theory since local biquandle (co)homology theory is an analogy of biquandle (co)homology theory.
Moreover through an isomorphism between two cohomology groups, we showed that Niebrzydowski's cocycle invariants and local biquandle cocycle invariants are the same.

Considering the main results shown in Section~\ref{sec:main} in this paper together with the results shown in \cite{NOO}, we have the following corollaries:
\begin{corollary}\label{cor:1}
Let $(B, X)$ be a shadow biquandle such that $X$ is strongly connected, 
$[\,] :X^3\to X$ the corresponding horizontal-tribracket of $(B, X)$, and $(X, \{ \uline{\star}\} , \{ \oline{\star}\})$ the corresponding local biquandle of $(B,X)$.
Let $A$ be an abelian group.
Then for any $n\in \mathbb Z$,  we have 
\[
H_n^{\rm SB} (B, X; A)  \cong H_n^{\rm LB} (X; A) \cong H_{n-1}^{\rm N} (X;A)
\]
and 
\[
H^n_{\rm SB} (B, X; A)  \cong H^n_{\rm LB} (X; A) \cong H^{n-1}_{\rm N} (X;A).
\]
\end{corollary}

\begin{corollary}\label{cor:2}
Let $L$ be a link.
Let $(B, X)$ be a shadow biquandle such that $X$ is strongly connected, 
$[\,] :X^3\to X$ the corresponding horizontal-tribracket of $(B, X)$, and $(X, \{ \uline{\star}\} , \{ \oline{\star}\})$ the corresponding local biquandle of $(B,X)$.
Then $$\mathcal{H}^{\rm SB}(L), ~~\mathcal{H}^{\rm LB}(L) ~~ \mbox{ and } ~~\mathcal{H}^{\rm N}(L)$$ are the same as link invariants.  
Moreover for an abelian group $A$, let $\theta \in Z^2_{\rm SB}(B, X;A)$, $\theta' \in Z^2_{\rm LB}(X;A)$ and $\bar{\theta} \in Z^1_{\rm N}(X;A)$ such that $\theta = \theta' \circ \mu_2$ and $\theta' = \bar\theta \circ \varphi_2$, where $Z^1_{\rm N}(X;A)$ is the first cocycle group of the Niebrzydowski's (co)homology theory and $\varphi_2$ is the bijective chain map defined in \cite{NOO}.
Then we have 
$$\Phi_{\theta}^{\rm SB}(L) =\Phi_{\theta'}^{\rm LB}(L)=\Phi_{\bar\theta}^{\rm N}(L).$$
\end{corollary}

\begin{corollary}\label{cor:3}
Let $F$ be a surface-link.
Let $(B, X)$ be a shadow biquandle such that $X$ is strongly connected, 
$[\,] :X^3\to X$ the corresponding horizontal-tribracket of $(B, X)$, and $(X, \{ \uline{\star}\} , \{ \oline{\star}\})$ the corresponding local biquandle of $(B,X)$.
Then $$\mathcal{H}^{\rm SB}(F), ~~\mathcal{H}^{\rm LB}(F) ~~ \mbox{ and } ~~\mathcal{H}^{\rm N}(F)$$ are the same as link invariants.  
Moreover for an abelian group $A$, let $\theta \in Z^3_{\rm SB}(B, X;A)$, $\theta' \in Z^3_{\rm LB}(X;A)$ and $\bar{\theta} \in Z^2_{\rm N}(X;A)$ such that $\theta = \theta' \circ \mu_3$ and $\theta' = \bar\theta \circ \varphi_3$, where $Z^2_{\rm N}(X;A)$ is the second cocycle group of the Niebrzydowski's (co)homology theory and $\varphi_3$ is the bijective chain map defined in \cite{NOO}.
Then we have 
$$\Phi_{\theta}^{\rm SB}(F) =\Phi_{\theta'}^{\rm LB}(F)=\Phi_{\bar\theta}^{\rm N}(F).$$
\end{corollary}


\subsection{Examples and Mochizuki's cocycles}\label{subsec:Mochizuki}

\begin{example}\label{ex:Mochizuki}
For a positive integer $n$ and an ideal $J$ of $\mathbb Z_n [t^{\pm 1}]$, let $(Q,X, \uline{*},*)$ is the shadow quandle, with
$Q=X=\mathbb Z_n [t^{\pm 1}]/J$ and $a\uline{*} b =a*b =ta+(1-t)b~(\forall a,b \in \mathbb Z_n [t^{\pm 1}]/J)$, defined in Example~\ref{ex:2}. Suppose that $1-t$ is a unit in $\mathbb Z_n [t^{\pm 1}]/J$.
Since $X$ is strongly connected and we have 
$$x \darrow y = (1-t)^{-1}(-tx+y) ~~~(x,y \in X), $$
it holds that 
\[
[x,y,z]= y*\big((x\darrow z) \oline{*} (x\darrow y) \big)= y*(x\darrow z) = y*\big((1-t)^{-1}(-tx+z)\big)= -tx + ty + z
\]
by Theorem~\ref{thm:1}, where we note that $a\oline{*} b = a ~(\forall a,b \in Q)$.
Thus Example~\ref{ex:4} is related to Example~\ref{ex:2}.
\end{example}

\begin{example}
Let $n$ be an odd prime number. 
Let $(R_n, X, \uline{*}, *)$ be the shadow quandle, with 
$R_n= X = \mathbb Z_n$ and $a\uline{*}b = a*b = 2b-a ~(\forall a,b \in \mathbb Z_n)$, defined in Example~\ref{ex:1}.
Since $X$ is strongly conneced and we have 
\[
x\darrow y = \frac{x+y}{2} ~~~(x,y\in X), 
\]
it holds that 
\[
[x,y,z]= y*\big((x\darrow z) \oline{*} (x\darrow y) \big)= y*(x\darrow z) = y*\Big(\frac{x+z}{2}\Big)= x-y+z
\]
by Theorem~\ref{thm:1}, where we note that $a\oline{*} b = a ~(\forall a,b \in R_n)$.
Thus Example~\ref{ex:3} is related to Example~\ref{ex:1}.

The shadow (bi)quandle $2$-cocycle $\theta_n: C_2^{\rm SB}(R_n, X) \to \mathbb Z_n$ defined by 
\[
\displaystyle \theta_n\big((x,y,z)\big) = (x-y) \frac{(2z-y)^n+y^n-2z^n}{n}
\] 
is called a {\it Mochizuki's cocycle}. Let $\theta_n^{\rm SB}=\theta_n$. The corresponding local biquandle $2$-cocycle $\theta_n^{\rm LB}$, i.e. $\theta^{\rm LB}_n \in Z^2_{\rm LB}(X;\mathbb Z_n)$ such that $\theta_n^{\rm SB} = \theta_n^{\rm LB} \circ \mu_2$, is defined by 
\begin{align*}
\theta_n^{\rm LB}\Big(\big((x,y), (x,z)\big)\Big) &= \theta_n^{\rm SB} \big( (x,x\darrow y, x \darrow z)\big) \\
&=\theta_n^{\rm SB} \Big( \big( x,\frac{x+y}{2}, \frac{x+z}{2} \big)\Big)\\
&=\Big(x-\frac{x+y}{2}\Big) \frac{ \Big(2 \frac{x+z}{2} - \frac{x+y}{2}\Big)^n +\Big( \frac{x+y}{2}\Big)^n-2 \Big( \frac{x+z}{2}\Big)^n  }{n}\\
&=\Big(\frac{x-y}{2}\Big) \frac{ \Big(\frac{x-y+2z}{2}\Big)^n +\Big( \frac{x+y}{2}\Big)^n-2 \Big( \frac{x+z}{2}\Big)^n   }{n}\\
&=2^{-2}(x-y) \frac{  (x-y+2z)^n +( x+y)^n- 2( x+z)^n   }{n}
\end{align*}
since $2^{-n}= 2^{-1}$ in $\mathbb Z_n$ by Fermat's little theorem, where the numerator is calculated in $\mathbb Z$ and it is divisible by $n$. We may define $\theta_n^{\rm LB}$ by 
\begin{align*}
\theta_n^{\rm LB}\Big(\big((x,y), (x,z)\big)\Big) &=(x-y) \frac{  (x-y+2z)^n +( x+y)^n- 2( x+z)^n   }{n}.
\end{align*}
Furthermore, the corresponding Niebrzydowski's $1$-cocycle $\theta^{\rm N}_n$, i.e. $\theta^{\rm N}_n \in Z^1_{\rm N}(X;\mathbb Z_n)$ such that $\theta^{\rm LB}_n = \theta^{\rm N}_n \circ \varphi_2$, is defined by 
\begin{align*}
\theta^{\rm N}_n\big((x, y, z)\big) &= \theta^{\rm LB}_n \Big(\big( (x, y), (x, \langle x,y,z \rangle )\big) \Big)\\
&= \theta^{\rm LB}_n \Big(\big( (x, y), (x, -x+y+z )\big) \Big) \\
&=2^{-2}(x-y) \frac{  \big(x-y+2(-x+y+z)\big)^n +\big( x+y\big)^n- 2\big( x+(-x+y+z)\big)^n   }{n}\\
&=2^{-2}(x-y) \frac{  (-x+y+2z)^n +( x+y)^n- 2( y+z)^n   }{n},
\end{align*}
where $\varphi_2$ is the bijective chain map defined in \cite{NOO}, $\langle\,  \rangle $ is the corresponding vertical-tribracket defined in \cite{NOO} and $\langle x,y,z \rangle= -x+y+z$ for this case, and where the numerator is calculated in $\mathbb Z$ and it is divisible by $n$. We may define $\theta^{\rm N}_n$ by 
\begin{align*}
\theta^{\rm N}_n\Big(\big(x, y, z\big)\Big) &=(x-y) \frac{  (-x+y+2z)^n +( x+y)^n- 2( y+z)^n   }{n}.
\end{align*}


The Mochizuki's cocycle $\theta_n$ induces a shadow biquadle $3$-cocycle $\theta_n^{\rm SB}$ by 
\[
\theta_n^{\rm SB}\big((x, y, z, w)\big) = \theta_n(y,z,w).
\]
The corresponding local biquandle $3$-cocycle $\theta_n^{\rm LB}$, i.e. $\theta^{\rm LB}_n \in Z^3_{\rm LB}(X;\mathbb Z_n)$ such that $\theta_n^{\rm SB} = \theta_n^{\rm LB} \circ \mu_3$, is defined by 
\begin{align*}
\theta_n^{\rm LB}\Big(\big((x,y), (x,z), (x,w)\big)\Big) &= \theta_n^{\rm SB} \big( (x,x\darrow y, x \darrow z, x \darrow w)\big) \\
&=\theta_n^{\rm SB} \Big( \big( x,\frac{x+y}{2}, \frac{x+z}{2} , \frac{x+w}{2} \big)\Big)\\
&=\theta_n \Big( \big( \frac{x+y}{2}, \frac{x+z}{2} , \frac{x+w}{2} \big)\Big)\\
&=2^{-2}(y-z) \frac{  (x-z+2w)^n +( x+z)^n- 2( x+w)^n   }{n},
\end{align*}
where the numerator is calculated in $\mathbb Z$ and it is divisible by $n$. We may define $\theta_n^{\rm LB}$ by 
\begin{align*}
\theta_n^{\rm LB}\Big(\big((x,y), (x,z), (z,w)\big)\Big) &=(y-z) \frac{  (x-z+2w)^n +( x+z)^n- 2( x+w)^n   }{n}.
\end{align*}
Furthermore, the corresponding Niebrzydowski's $2$-cocycle $\theta^{\rm N}_n$, i.e. $\theta^{\rm N}_n \in Z^2_{\rm N}(X;\mathbb Z_n)$ such that $\theta^{\rm LB}_n = \theta^{\rm N}_n \circ \varphi_3$, is defined by 
\begin{align*}
\theta^{\rm N}_n\big((x, y, z, w)\big) &= \theta^{\rm LB}_n \Big(\big( (x, y), (x, \langle x,y,z \rangle ), (x, \langle x,y,\langle  y,z,w \rangle \rangle ) \big) \Big)\\
&= \theta^{\rm LB}_n \Big(\big( (x, y), (x, -x+y+z ), (x, -x+z+w )\big) \Big) \\
&=2^{-2}(x-z) \frac{  (-y+z+2w)^n +( y+z)^n- 2( z+w)^n   }{n},
\end{align*}
where $\varphi_3$ is the bijective chain map defined in \cite{NOO},  $\langle\,  \rangle $ is the corresponding vertical-tribracket defined in \cite{NOO} and $\langle x,y,z \rangle= -x+y+z$ for this case, and where the numerator is calculated in $\mathbb Z$ and it is divisible by $n$. We may define $\theta^{\rm N}_n$ by 
\begin{align*}
\theta^{\rm N}_n\Big(\big(x, y, z, w\big)\Big) &=(x-z) \frac{  (-y+z+2w)^n +( y+z)^n- 2( z+w)^n   }{n}.
\end{align*}

\end{example}

\section*{Acknowledgments}

The author wishes to express her thanks to Natsumi Oyamaguchi for several helpful comments.

The author was supported by JSPS KAKENHI Grant Number 16K17600.


\end{document}